\documentclass[11pt]{amsart}
\usepackage{epigraph}
\usepackage{amsthm}
\usepackage{bm}
\usepackage{tikz-cd}
\usepackage{amsmath}
\usepackage{amssymb}
\usepackage{relsize}
\usepackage{amsfonts}
\usepackage{amsxtra}
\usepackage{braket}
\usepackage{lmodern}
\usepackage{mathrsfs}
\usepackage{array}
\usepackage{xcolor}
\definecolor{citecol}{rgb}{0.07,0.07,0.05}
\definecolor{urlcol}{rgb}{0.06,0.04,0.09}
\definecolor{linkcol}{rgb}{0.01,0.03,0.08}
\usepackage[colorlinks,linkcolor=linkcol,urlcolor=urlcol,citecolor=citecol,pagebackref,breaklinks]{hyperref}
\usepackage{appendix} 
\usepackage[lite,abbrev,msc-links,alphabetic]{amsrefs}
\usepackage{indentfirst}
\usepackage{mathtools}
\usepackage[all]{xy}
\allowdisplaybreaks
 \numberwithin{equation}{subsection}
\theoremstyle{plain}
\newtheorem{theorem}{Theorem}[section]
\newtheorem{lemma}[theorem]{Lemma}
\newtheorem{proposition}[theorem]{Proposition}

\newtheorem{conjecture}[theorem]{Conjecture}

\theoremstyle{definition}
\newtheorem{definition}[theorem]{Definition}
\newtheorem{remark}[theorem]{Remark}

\newtheorem*{acknowledgement}{Acknowledgement}

\theoremstyle{remark}

\newcommand{\BC}{{\mathbb C}}

\newcommand{\BF}{{\mathbb F}}

\newcommand{\BL}{{\mathbb L}}

\newcommand{\BP}{{\mathbb P}}
\newcommand{\BQ}{{\mathbb Q}}

\newcommand{\BV}{{\mathbb V}}

\newcommand{\BX}{{\mathbb X}}
\newcommand{\BY}{{\mathbb Y}}
\newcommand{\BZ}{{\mathbb Z}}

\newcommand{\CA}{{\mathcal A}}
\newcommand{\CB}{{\mathcal B}}

\newcommand{\CF}{{\mathcal F}}
\newcommand{\CG}{{\mathcal G}}

\newcommand{\CK}{{\mathcal K}}

\newcommand{\CN}{{\mathcal N}}

\newcommand{\CR}{{\mathcal R}}
\newcommand{\CS}{{\mathcal S}}

\newcommand{\CY}{{\mathcal Y}}
\newcommand{\CZ}{{\mathcal Z}}

\newcommand{\FA}{{\mathfrak A}}
\newcommand{\FB}{{\mathfrak B}}
\newcommand{\FC}{{\mathfrak C}}
\newcommand{\FD}{{\mathfrak D}}

\newcommand{\FM}{{\mathfrak M}}

\newcommand{\FX}{{\mathfrak X}}

\newcommand{\Fa}{{\mathfrak a}}
\newcommand{\Fb}{{\mathfrak b}}

\DeclareMathOperator{\Lie}{Lie}
\DeclareMathOperator{\Charpol}{Charpol}
\DeclareMathOperator{\End}{End}

\DeclareMathOperator{\Spf}{Spf}

\DeclareMathOperator{\Ker}{Ker}

\DeclareMathOperator{\Hom}{Hom}

\DeclareMathOperator{\Nm}{Nm}

\DeclareMathOperator{\val}{val}

\DeclareMathOperator{\Tr}{Tr}
\DeclareMathOperator{\Mod}{mod}
\DeclareMathOperator{\diag}{diag}

\author{Sungyoon Cho}
\address[Sungyoon Cho]{Department of Mathematics, University of Arizona}
\email{sungyooncho@math.arizona.edu}

\title[Representation densities]{On local representation densities of hermitian forms and special cycles}
\date{\today}

\begin{document}

\begin{abstract}
In this paper, we reformulate conjectural formulas for the arithmetic intersection numbers of special cycles on unitary Shimura varieties with minuscule parahoric level structure in terms of weighted counting of lattices containing special homomorphisms.
\end{abstract}

\maketitle
\tableofcontents{}

\section{Introduction}
In \cite{KR2} and \cite{KR3}, Kudla and Rapoport made a conjectural formula, the so-called Kudla-Rapoport conjecture, on the arithmetic intersection numbers of special cycles on unitary Shimura varieties with hyperspecial level structure. This conjecture was proved by Li and Zhang in \cite{LZ}. One of key ingredients in the proof of the Kudla-Rapoport conjecture was a reformulation of the derivative of a representation density in terms of weighted counting of lattices containing special homomorphisms. This is the work of \cite{CY} in the case of orthogonal group and its unitary variant was used to prove the conjecture in \cite{LZ}. Also, a variant of this formula in ramified field extension case was used in \cite{LL} to prove an analogue of the Kudla-Rapoport conjecture for exotic good reduction case.

In \cite{Cho2}, we made a conjectural formula for the arithmetic intersection number of special cycles on unitary Shimura varieties with minuscule parahoric structure. This was formulated in terms of certain weighted representation densities.

In this paper, we reformulate this formula in terms of weighted counting of lattices. Now, let us describe our result in more detail.

Let $p$ be an odd prime and let $F$ be an unramified extension of $\BQ_p$ with ring of integers $O_F$ and residue field $\BF_q$. Let $\pi$ be a uniformizer (this is $p$, but we use the general notation for future use). Let $E$ be a quadratic unramified extension of $F$ with ring of integers $O_E$. Let $\breve{E}$ be the completion of a maximal unramified extension of $E$.
For integers $0 \leq h \leq n$, we can consider the unitary Rapoport-Zink space $\CN^h_{E/F}(1,n-1)$ over $\Spf O_E$. This uniformizes the basic locus of the Rapoport-Smithling-Zhang unitary Shimura varieties (\cite{RSZ2}) with minuscule parahoric level structure. This is a regular formal scheme and hence we can do intersection theory on this scheme.

There is a hermitian space $\BV$ over $E$ with hermitian form $h(\cdot,\cdot)$ which is called the space of special homomorphisms. For each special homomorphism $x$ in $\BV$, one can attach Cartier divisors $\CZ(x)$ and $\CY(y)$ in $\CN^h(1,n-1)$ that are called the Kudla-Rapoport divisors or special cycles. We are interested in the arithmetic intersection numbers of these special cycles on $\CN^h(1,n-1)$.

By \cite[Proposition 5.11]{Cho} and \cite[Remark 2.7]{Cho2}, for the arithmetic intersection numbers of special cycles in $\CN^h(1,n-1)$, it suffices to consider the ones in $\CN^n(1,2n-1)$. In \cite[Conjecture 3.17, 3.26]{Cho2}, we made a conjectural formula for these intersection numbers in terms of weighted representation densities as follows.

\begin{conjecture}\label{conjecture1.1}(\cite[Conjecture 3.17, 3.26]{Cho2})\\
 For a basis $\lbrace x_1, \dots, x_{2n-m}, y_1, \dots, y_m \rbrace$ of $\BV$, and special cycles $\CZ(x_1)$, $\dots$, $\CZ(x_{2n-m})$, and $\CY(y_1)$, $\dots$, $\CY(y_m)$ in $\CN^n(1,2n-1)$, we have
		\begin{equation*}
			\begin{array}{l}
				\langle \CZ(x_1),\dots,\CZ(x_{2n-m}),\CY(y_1),\dots,\CY(y_m) \rangle\\:=\chi(O_{\CZ(x_1)}\otimes^{\BL}\dots\otimes^{\BL}O_{\CY(y_k)})\\
				=\dfrac{1}{W_{n,n}(A_n,0)}\lbrace W'_{m,n}(B,0)-\mathlarger{\sum}_{0 \leq i \leq n-1} \beta_i^mW_{m,i}(B,0)\rbrace.
			\end{array}
		\end{equation*}
		Here $\chi$ is the Euler-Poincare characteristic and $\otimes^{\BL}$ is the derived tensor product. Also, $B$ is the matrix
		\begin{equation*}
			B=\left(\begin{array}{cc} 
				h(x_i,x_j) & h(x_i,y_l)\\
				h(y_k,x_j)& h(y_k,y_l)
			\end{array} 
			\right)_{1\leq i,j \leq 2n-m, 1\leq k,l \leq m}.
		\end{equation*}
\end{conjecture}

In the present paper, we only consider $m=0$ case, i.e., the arithmetic intersection number of only $\CZ$-cycles. It is because, in this case, the above weighted representation densities are usual representation densities and hence they depend only on the $O_E$-lattice generated by $\lbrace x_1, \dots, x_{2n}\rbrace$. For general $h$, we need to modify these lattices in some way and hence this will be postponed to our future work.

Let $A$ be a $2n\times 2n$ hermitian matrix and let $B$ be the hermitian matrix of special homomorphisms $x_1, \dots, x_{2n}$. Let $L_{x}$ be the $O_E$-lattice generated by $x_1, \dots, x_{2n}$. Then one can regard the representation density $\alpha(A,B)/\alpha(A,A)$ as counting of different lattices $L'$ with hermitian forms $A$ containing $L_{x}$ (see Section \ref{subsection3.2}). Therefore, we want to express \begin{equation*}\dfrac{1}{W_{n,n}(A_n,0)}\lbrace W'_{0,n}(B,0)-\mathlarger{\sum}_{0 \leq i \leq n-1} \beta_i^0W_{0,i}(B,0)\rbrace \end{equation*}
as a linear sum $\mathlarger{\sum}_{A} (\text{constant}) \alpha(A,B)/\alpha(A,A)$. Let $\CR_{2n}^{0+}$ be the set
\begin{equation*}
	\CR^{0+}_{2n}:=\lbrace A_{\lambda} \vert \lambda=(\lambda_1,\dots, \lambda_{2n}), \lambda_i \in \BZ, \lambda_1 \geq \dots \geq \lambda_{2n} \geq 0 \rbrace,
\end{equation*}
and let $A_{\lambda}=\diag(\pi^{\lambda_1},\dots,\pi^{\lambda_{2n}})$. Note that these $A_{\lambda}$'s form a complete set of representatives of $GL_{2n}(O_E)$-equivalent classes of nondegenerate hermitian matrices. In Section \ref{subsection3.1}, we prove some linear independence and this implies that there is a unique way to express the above formula as a linear sum of representation densities
\begin{equation*}
	\mathlarger{\sum}_{\lambda \in \CR_{2n}^{0+}} (\text{constants}) \dfrac{\alpha(A_{\lambda},B)}{\alpha(A_{\lambda},A_{\lambda})}.
\end{equation*}

In Section \ref{section4}, we compute these constants and our theorem can be formulated as follows.

\begin{theorem}(Theorem \ref{theorem4.7}) We have
	\begin{equation*}\begin{array}{l}\dfrac{1}{W_{n,n}(A_n,0)}\lbrace W'_{0,n}(B,0)-\mathlarger{\sum}_{0 \leq i \leq n-1} \beta_i^0W_{0,i}(B,0)\rbrace\\
	=\mathlarger{\sum}_{\lambda \in \CR_{2n}^{0+}} D_{\lambda} \dfrac{\alpha(A_{\lambda},B)}{\alpha(A_{\lambda},A_{\lambda})}-\mathlarger{\sum}_{0 \leq i \leq n-1} \Fb_i^0\dfrac{\alpha(A_{(1^i,0^{2n-i})},B)}{\alpha(A_{(1^i,0^{2n-i})},A_{(1^i,0^{2n-i})})} \\\\
	
	=\mathlarger{\sum}_{\lambda \in \CR_{2n}^{0+}}\mathlarger{\sum}_{L' \in A_{\lambda}} D_{\lambda}1_{L'}(x_1,\dots,x_{2n})
	-\mathlarger{\sum}_{0 \leq i \leq n-1} \mathlarger{\sum}_{L' \in A_{(1^i,0^{2n-i})}} \Fb_i^0 1_{L'}(x_1, \dots, x_{2n}).	
	\end{array} \end{equation*}
\end{theorem}
We refer to Theorem \ref{theorem4.7} for $D_{\lambda}$ and Section \ref{subsection4.2} for $\Fb_i^0$. We emphasize that in the case of hyperspecial level structure $\CN^0(1,n-1)$ the arithmetic intersection number is
\begin{equation*}
	\langle \CZ(x_1),\dots,\CZ(x_n)\rangle=\dfrac{\alpha'(1_n,B)}{\alpha(1_n,1_n)}=\mathlarger{\sum}_{\lambda \in \CR_{n}^{0+}} C_{\lambda} \dfrac{\alpha(A_{\lambda},B)}{\alpha(A_{\lambda},A_{\lambda})},
\end{equation*}
and the constants $C_{\lambda}$ depend only on $n$, the parity of $\sum_i \lambda_i$, and the number of zeros in $\lambda$. In minuscule parahoric level structure cases, the constants $D_{\lambda}$ depend only on $n$, the parity of $\sum_i \lambda_i$, the number of zeros in $\lambda$, and the number of 1's in $\lambda$.

The above theorem implies that the above conjecture \ref{conjecture1.1} can be written as follows.

\begin{conjecture}(Conjecture \ref{conjecture4.9})
	For a basis $\lbrace x_1, \dots, x_{2n} \rbrace$ of $\BV$, and special cycles $\CZ(x_1)$, $\dots$, $\CZ(x_{2n})$ in $\CN^n(1,2n-1)$, we have
	\begin{equation*}\begin{array}{l}
			\langle \CZ(x_1),\dots,\CZ(x_{2n})\rangle\\
			=\mathlarger{\sum}_{\lambda \in \CR_{2n}^{0+}} D_{\lambda} \dfrac{\alpha(A_{\lambda},B)}{\alpha(A_{\lambda},A_{\lambda})}-\mathlarger{\sum}_{0 \leq i \leq n-1} \Fb_i^0\dfrac{\alpha(A_{(1^i,0^{2n-i})},B)}{\alpha(A_{(1^i,0^{2n-i})},A_{(1^i,0^{2n-i})})} \\\\
			
			=\mathlarger{\sum}_{\lambda \in \CR_{2n}^{0+}}\mathlarger{\sum}_{L' \in A_{\lambda}} D_{\lambda}1_{L'}(x_1,\dots,x_{2n})
			-\mathlarger{\sum}_{0 \leq i \leq n-1} \mathlarger{\sum}_{L' \in A_{(1^i,0^{2n-i})}} \Fb_i^0 1_{L'}(x_1, \dots, x_{2n}).
		\end{array}
	\end{equation*}
\end{conjecture}

\begin{acknowledgement}
	I would like to thank Chao Li and Yifeng Liu for helpful discussions.
\end{acknowledgement}

\section{Weighted representation densities and conjectures}\label{section2}
In this section, we will recall the definition of weighted representation densities and their formulas from \cite{Cho2} and prove some properties. Also, we will recall the conjectural formula in \cite[Conjecture 3.17, 3.26]{Cho2}.

\subsection{Weighted representation densities}\label{subsection2.1}
In this subsection, we will recall the definition of weighted representation densities and formulas from \cite[Section 3.1]{Cho2}.

We fix a prime $p>2$. Let $F$ be a finite extension of $\BQ_p$ with ring of integers $O_F$, and residue field $\BF_q$. Let $\pi$ be a fixed uniformizer of $O_F$. Let $E$ be a quadratic unramified extension of $F$ with ring of integers $O_E$. We denote by $^*$ the nontrivial Galois automorphism of $E$ over $F$.

We fix the standard additive character $\psi:F \rightarrow \BC^{\times}$ that is trivial on $O_F$. Let $V^{+}$ (resp. $V^{-}$) be a split (resp. nonsplit) $2n$-dimensional hermitian vector space over $E$ and let $\CS((V^{\pm})^{2n})$ be the space of Schwartz functions on $(V^{\pm})^{2n}$. We define $V_{r,r}$ to be the split hermitian space of signature $(r,r)$ and let $L_{r,r}$ be a self-dual lattice in $V_{r,r}$. We denote by $\varphi_{r,r}$ the characteristic function of $(L_{r,r})^{2n}$ and we denote by $(V^{\pm})^{[r]}$ the space $V^{\pm} \otimes V_{r,r}$. For any function $\varphi \in \CS((V^{\pm})^{2n})$, we define $\varphi^{[r]}$ by the function $\varphi \otimes \varphi_{r,r} \in \CS(((V^{\pm})^{[r]})^{2n})$.

We denote by $\Gamma_{2n}$ the Iwahori subgroup
\begin{equation*}
	\Gamma_{2n}=\lbrace \gamma=(\gamma_{ij})\in GL_{2n}(O_E)\vert \gamma_{ij} \in \pi O_E \text{ if }i>j \rbrace.
\end{equation*}

We define the sets
\begin{equation*}
	\begin{array}{l}
		V_{2n}(E)=\lbrace Y \in M_{2n,2n}(E) \vert ^tY^*=Y\rbrace,\\
		X_{2n}(E)=\lbrace X \in GL_{2n}(E) \vert ^tX^*=X \rbrace.
	\end{array}
\end{equation*}

For $g \in GL_{2n}(E)$ and $X \in X_{2n}(E)$, we define the group action of $GL_{2n}(E)$ on $X_{2n}(E)$ by $g\cdot X=gX^tg^*$.

For $X,Y \in V_{2n}(E)$, we denote by $\langle X,Y \rangle=Tr(XY)$.

For $X \in M_{m,n}(E)$ and $A \in V_m(E)$, we denote by $A[X]=(   ^tX^*AX)$.

\begin{definition}
	(\cite[definition3.1]{Cho2}) Let $0 \leq h,t \leq n$. If $t$ is even (resp. odd) we define $L_t$ as a lattice of rank $2n$ in $V^+$ (resp. $V^-$) with hermitian form
	\begin{equation*}
		A_t:=\left( \begin{array}{ll}
			1_{2n-t} &\\
			& \pi^{-1}1_t
		\end{array}\right).
	\end{equation*}

Let $1_{h,t} \in \CS((V^{\pm})^{2n})$ be the characteristic function of $(L_t^{\vee})^{2n-h} \times L_t^h$, where $L_t^{\vee}$ is the dual lattice of $L_t$ with respect to the hermitian form.

For an element $B \in X_{2n}(E)$, we define
\begin{equation*}
	W_{h,t}(B,r):=\int_{V_{2n}(E)}\int_{M_{2n+2r,2n}(E)} \psi(\langle Y, A_t^{[r]}[x]-B \rangle) 1_{h,t}^{[r]}(X) dXdY.
\end{equation*}

Here $A_t^{[r]}$ is
\begin{equation*}
	A_t^{[r]}=\left( \begin{array}{ll}
		A_t & \\
		& 1_{2r}
	\end{array}\right),
\end{equation*}
and $dY$ (resp. $dX$) is the Haar measure on $V_{2n}(E)$ (resp. $M_{2n+2r,2n}(E)$) such that
\begin{equation*}
	\int_{V_{2n}(O_E)} dY=1 \text{ (resp. }\int_{M_{2n+2r,2n}(O_E)} dX=1).
\end{equation*}
\end{definition}

We have the following formula for $W_{h,t}(B,r)$.

\begin{lemma}\label{lemma2.2}
	(\cite[Lemma 3.3]{Cho2}) For $B \in X_{2n}(E)$, we have
	\begin{equation*}
		W_{h,t}(B,r)=\mathlarger{\sum}_{Y \in \Gamma_{2n} \backslash X_{2n}(E)} \dfrac{\CG(Y,B)\CF_h(Y,A_t^{[r]})}{\alpha(Y;\Gamma_{2n})}.
	\end{equation*}

	Here, we define
	\begin{equation*}
		\CF_h(Y,A_t^{[r]})=\int_{M_{2n+2r,2n}(E)} \psi(\langle Y, A_t^{[r]}[X]\rangle)1_{h,t}^{[r]}(X)dX,
	\end{equation*}
	and
	\begin{equation*}
		\CG(Y,B)=\int_{\Gamma_{2n}}\psi(\langle Y, -B[\gamma]\rangle)d\gamma,
	\end{equation*}
	where $d\gamma$ is the Haar measure on $M_{2n,2n}(O_{E})$ such that $\int_{M_{2n, 2n}(O_{E})} d\gamma=1$.
	
	Also, we define
	\begin{equation*}
		\alpha(Y; \Gamma_{2n})=\lim_{d \rightarrow \infty} q^{-4dn^2} N_d(Y;\Gamma_{2n}),
	\end{equation*}
	where
	\begin{equation*}
		N_d(Y;\Gamma_{2n})=\vert \lbrace \gamma \in \Gamma_{2n} (\Mod \pi^d) \vert \gamma \cdot Y \equiv Y (\Mod \pi^d) \rbrace \vert.
	\end{equation*}
\end{lemma}

\begin{definition}
	One can regard $\CF_h(Y,A_t^{[r]})$ and $W_{h,t}(B,r)$ as functions of $X=(-q)^{-2r}$. We define
	\begin{equation*}
		\CF_h'(Y,A_t^{[0]}):=-\dfrac{d}{dX}\CF_h(Y,A_t^{[r]})\vert_{X=1}.
	\end{equation*}
Also, we define
\begin{equation*}
	W_{h,t}'(B,0):=-\dfrac{d}{dX}W_{h,t}(B,r)\vert_{X=1}.
\end{equation*}
\end{definition}

\subsection{Special cycles}\label{subsection2.2}

In this subsection, we will recall some facts and definitions of unitary Rapoport-Zink spaces $\CN^h_{E/F}(1,n-1)$ and special cycles.

Here, we assume that $F$ is an unramified finite extension of $\BQ_p$. Let $\breve{E}$ be the completion of a maximal unramified extension of $E$ and let $O_{\breve{E}}$ be its ring of integers. We denote by $k$ the residue field of $O_{\breve{E}}$.

Let $h$ be an integer such that $0 \leq h \leq n$. To define $\CN^h_{E/F}(1,n-1)$, we need to fix a triple $(\BX,i_{\BX},\lambda_{\BX})$ consisting of the following data:

\begin{enumerate}
	\item $\BX$ is a supersingular strict formal $O_F$-module of $F$-height 2n over $\BF_{q^2}$. We refer to \cite[Definition 2.1]{Cho} for complete details.
	
	\item $i_{\BX}:O_E \rightarrow \End \BX$ is an $O_E$-action on $\BX$ that extends the $O_F$-action on $\BX$ and it satisfies the following signature condition:
	For all $a\in O_E$
	\begin{equation*}
		\Charpol(i_{\BX}(a)\vert \Lie(\BX))=(T-a)(T-a^*)^{n-1}.
	\end{equation*}

	\item $\lambda_{\BX}$ is a polarization
	\begin{equation*}
		\lambda_{\BX}:\BX \rightarrow \BX^{\vee},
	\end{equation*}
such that the corresponding Rosati involution induces the involution $*$ on $O_E$. Also, we assume that $\Ker \lambda_{\BX} \subset \BX[\pi]$ and its order is $q^{2h}$.
\end{enumerate}

This triple $(\BX,i_{\BX},\lambda_{\BX})$ is called the framing object.

Now, let $(Nilp)$ be the category of $O_E$-schemes $S$ such that $\pi$ is locally nilpotent on $S$. Then, we can define the functor $\CN^h_{E/F}(1,n-1)$ sending each $S \in (Nilp)$ to the set of isomorphism classes of tuples $(X,i_X,\lambda_X,\rho_X)$. Here $X$ is a supersingular strict formal $O_F$-module, $i_X$ is an $O_E$-action, $\lambda_X$ is a polarization, and $\rho_X$ is a certain $O_E$-linear quasi-isogeny. We refer \cite[Definition 2.1]{Cho} for complete details.

The functor $\CN^h_{E/F}(1,n-1) \otimes O_{\breve{E}}$ is representable by a formal scheme over $\Spf O_{\breve{E}}$ that is locally formally of finite type. Also, this formal scheme is regular.\\

Now, we recall the definitions of special cycles from \cite[Section 5]{Cho}. Let  $(\overline{\BY},i_{\overline{\BY}},\lambda_{\overline{\BY}})$ be the framing object of $\CN^0_{E/F}(0,1)$.

\begin{definition} Let $\CN^0=\CN^0_{E/F}(0,1) \otimes O_{\breve{E}}$, $\CN=\CN^h_{E/F}(1,n-1)\otimes O_{\breve{E}}$, and $\widehat{\CN}=\CN^{n-h}_{E/F}(1,n-1)\otimes O_{\breve{E}}$.
	\begin{enumerate}
		\item (\cite[Definition 3.1]{KR2}) We define the space of special homomorphisms as the $E$-vector space
		\begin{equation*}
			\BV:=\Hom_{O_E}(\overline{\BY},\BX) \otimes_{\BZ} \BQ.
		\end{equation*}
	We define a hermitian form $h$ on $\BV$ as
	\begin{equation*}
		\forall x \text{ and }y, \text{ }h(x,y)=\lambda_{\overline{\BY}}^{-1} \circ y^{\vee} \circ \lambda_{\BX} \circ x \in \End_{O_E}{\overline{\BY}} \otimes \BQ\simeq E.
	\end{equation*}

\item We denote by $\theta:\CN \rightarrow \widehat{\CN}$ the isomorphism which is defined in \cite[Remark 5.2]{Cho}.\\

\item For $x \in \BV$, we define the special cycles $\CZ(x)$ to be the closed formal subscheme of $\CN^0 \times \CN$ with the following lifting property: 

For each $O_{\breve{E}}$-scheme $S$ such that $\pi$ is locally nilpotent, $\CZ(x)(S)$ is the set of all points $\eta=(\overline{Y},i_{\overline{Y}},\lambda_{\overline{Y}},\rho_{\overline{Y}},X,i_X,\lambda_X,\rho_X)$ in $\CN^0 \times \CN$ such that $\rho_{X}^{-1} \circ x \circ \rho_{\overline{\BY}}$ extends to a homomorphism from $\overline{Y}$ to $X$.\\

\item For each $y \in \BV$, we define the special cycles $\CY(y)$ in $\CN^0_{E/F}(0,1) \times \CN^h_{E/F}(1,n-1)$ as follows. First, we consider the special cycle $\CZ(\lambda_{\BX}\circ y) \in \CN^0 \times \widehat{\CN}$. Then we define $\CY(y)$ as $(id \times \theta)^{-1}(\CZ(\lambda_{\BX}\circ y))$ in $\CN^0 \times \CN$.\\

Note that $\CN^0_{E/F}(0,1)$ can be identified with $\Spf O_{\breve{E}}$, and hence we can regard $\CZ(x)$ and $\CY(y)$ as closed formal subschemes of $\CN$.
	\end{enumerate} 
\end{definition}

\begin{proposition}
	(\cite[Proposition 5.9]{Cho}) For $x,y \in \BV\backslash \lbrace 0 \rbrace,$ $\CZ(x)$ and $\CY(y)$ are Cartier divisors in $\CN^0 \times \CN$ (or empty).
\end{proposition}

\subsection{Conjectures on the arithmetic intersection numbers of special cycles}\label{subsection2.3} In this subsection, we will recall the conjectural formula for the arithmetic intersection numbers of special cycles in \cite[Conjecture 3.17, 3.26]{Cho2}.

First, we recall \cite[Theorem 3.16]{Cho2}.

\begin{proposition}
	(\cite[Theorem 3.16]{Cho2}) There are unique constants
	\begin{equation*}
		\beta_0^h,\dots,\beta^h_{n-1},\beta_0^{2n-h},\dots,\beta_{n-1}^{2n-h},\delta_h,
	\end{equation*}
such that
	\begin{equation*}
		\begin{array}{l}
			W_{h,n}'(B,0)-W_{2n-h,n}'(B^{\vee_h},0)\\
			=\mathlarger{\sum}_{0 \leq i \leq n-1} \beta_i^hW_{h,i}(B,0)
			-\mathlarger{\sum}_{0 \leq j \leq n-1}\beta_j^{2n-h}W_{2n-h,j}(B^{\vee_h},0)
			+\delta_hW_{h,n}(B,0).
		\end{array}
	\end{equation*}

Here, for a matrix
\begin{equation*}
	B=\left(\begin{array}{ll}
		A & B\\
		C & D
	\end{array}\right),
\end{equation*}
in $X_{2n}(E)$, where 
\begin{equation*}
	\begin{array}{l}
		A \in M_{2n-h,2n-h}(E);\\
		B \in M_{2n-h,h}(E);\\
		C \in M_{h,2n-h}(E);\\
		D \in M_{h,h}(E),
	\end{array}
\end{equation*}
we denote by $B^{\vee_h}$ the matrix
\begin{equation*}
		B^{\vee_h}=\left(\begin{array}{ll}
		\pi D & C\\
		B & \pi^{-1}A
	\end{array}\right)
\end{equation*}
\end{proposition}

In the proof of \cite[Theorem 3.16]{Cho2}, we showed that the constants
\begin{equation*}
	\beta_0^h,\dots,\beta^h_{n-1},\beta_0^{2n-h},\dots,\beta_{n-1}^{2n-h},\delta_h,
\end{equation*}
satisfies the following matrix relation.

	\begin{equation*}
	\FB\left(\begin{array}{c}
		\beta_0^h\\
		\vdots\\
		\beta_{n-1}^h\\
		-\beta_{0}^{2n-h}\\
		\vdots\\
		-\beta_{n-1}^{2n-h}\\
		\delta_h
	\end{array}
	\middle)=(-q)^{-2n(2n-h)}\middle(\begin{array}{c}
		-(2n-h)\\
		\vdots\\
		-1\\
		0\\
		1\\
		\vdots\\
		h
	\end{array}\right)
\end{equation*}

where
	\begin{equation*}
	\begin{array}{l}
		m_{it}=(-q)^{(n-t)(i-(2n-h+1))-2t(2n-h)}\\
		n_{it}=q^{-(2n-h)^2+h^2}(-q)^{-(n-t)(i-(2n-h+1))-2th},
	\end{array}
\end{equation*}
and $\FB$ is the following $(2n+1) \times (2n+1)$ matrix.
	\begin{equation*}
	\begin{array}{l}
		\FB=\\\left(\begin{array}{ccccccc}
			m_{10}& \dots& m_{1(n-1)}& n_{10}& \dots& n_{1(n-1)} & m_{1n}\\
			\vdots&\vdots&\vdots&\vdots&\vdots&\vdots&\vdots\\
			m_{(2n+1)0}& \dots &m_{(2n+1)(n-1)} & n_{(2n+1)0} & \dots & n_{(2n+1)(n-1)}&m_{2n+1,n}
		\end{array}\right)
	\end{array}
\end{equation*}

We know that $\FB$ is invertible, and hence we can compute

\begin{equation*}
	\left(\begin{array}{c}
		\beta_0^h\\
		\vdots\\
		\beta_{n-1}^h\\
		-\beta_{0}^{2n-h}\\
		\vdots\\
		-\beta_{n-1}^{2n-h}\\
		\delta_h
	\end{array}
	\middle)=(-q)^{-2n(2n-h)}\FB^{-1}\middle(\begin{array}{c}
		-(2n-h)\\
		\vdots\\
		-1\\
		0\\
		1\\
		\vdots\\
		h
	\end{array}\right)
\end{equation*}

Now, let us compute these constants explicitly.

\begin{proposition}\label{proposition2.7}

We have
	\begin{equation*}
		\beta^h_i=\alpha_{i+1,h}^{-1}\mathlarger{\mathlarger{\Biggl(}}\dfrac{\mathlarger{\prod}_{1 \leq m \leq 2n, m \neq i+1} (1-x_m)}{\mathlarger{\prod}_{1 \leq m \leq 2n+1, m \neq i+1}(x_m-x_{i+1})}\mathlarger{\mathlarger{\Biggl)}},
	\end{equation*}
where
	\begin{equation*}
	\begin{array}{ll}
		\alpha_{i,h}=(-q)^{(n+1-i)(2n-h)},& 1 \leq i \leq n;\\
		\alpha_{i,h}=(-q)^{(2n+1-i)(2n+h)},& n+1 \leq i \leq 2n;\\
		\alpha_{2n+1,h}=1,
	\end{array}
\end{equation*}
and
\begin{equation*}
	\begin{array}{ll}
		x_i=(-q)^{n+1-i},& 1 \leq i \leq n; \\
		x_i=(-q)^{i-2n-1},& n+1 \leq i \leq 2n;\\
		x_{2n+1}=1.
	\end{array}
\end{equation*}
	
\end{proposition}
\begin{proof}
	Indeed, this computation was almost done in \cite[Proposition A.1]{Cho2}. Let $\FX$ be the following Vandermonde matrix.

	\begin{equation*}
	\FX=\left( \begin{array}{cccc}
		1 & 1& \dots  & 1\\
		x_1 &x_2 &    &x_{2n+1}\\
		x_1^2&x_2^2 &  & x_{2n+1}^2\\
		\vdots & & \ddots & \vdots\\
		x_1^{2n} & x_2^{2n} & \dots & x_{2n+1}^{2n}
	\end{array}\right).
\end{equation*}

Also, we define a diagonal matrix $\Fa_h$ as
	\begin{equation*}
	\Fa_h=\left( \begin{array}{cccc}
		\alpha_{1,h} & & & \\
		 &\alpha_{2,h} &  &\\
		 & & \ddots &\\
		 &  &  & \alpha_{2n+1,h}
	\end{array}\right).
\end{equation*}

Then, we can check that $(-q)^{2n(2n-h)}\FB=\FX\Fa_h.$

Therefore, we have

\begin{equation*}
	\left(\begin{array}{c}
		\beta_0^h\\
		\vdots\\
		\beta_{n-1}^h\\
		-\beta_{0}^{2n-h}\\
		\vdots\\
		-\beta_{n-1}^{2n-h}\\
		\delta_h
	\end{array}
	\middle)=\Fa_h^{-1}\FX^{-1}\middle(\begin{array}{c}
		-(2n-h)\\
		\vdots\\
		-1\\
		0\\
		1\\
		\vdots\\
		h
	\end{array}\right)
\end{equation*}

Let $\FX^{-1}=(y_{ij})$. Since $\FX$ is a Vandermonde matrix, we know that
	\begin{equation*}
	y_{ij}=\left\lbrace \begin{array}{cl}
		\dfrac{(-1)^{j-1}\mathlarger{\mathlarger{\sum}}_{\substack{1 \leq m_1 < \dots <m_{2n+1-j} \leq 2n+1\\ m_1, \dots,m_{2n+1-j} \neq i}}x_{m_1} \dots x_{m_{2n+1-j}}}{\mathlarger{\mathlarger{\prod}}_{1 \leq m \leq 2n+1, m \neq i}(x_m-x_i)} &, 1 \leq j <2n+1\\
		&\\
		\dfrac{1}{\mathlarger{\mathlarger{\prod}}_{1 \leq m \leq 2n+1, m \neq i}(x_m-x_i)}& ,j=2n+1.
	\end{array}\right.
\end{equation*}

Note that $y_{ij}$ can be regarded as the $z^{2n+1-j}$-coefficient of
\begin{equation*}
	\mathlarger{\prod}_{1 \leq m \leq 2n+1, m \neq i} \dfrac{(1-x_mz)}{(x_m-x_i)}.
\end{equation*}

Also, we have $\beta_i^h=\alpha_{i+1,h}^{-1}\mathlarger{\sum}_{1 \leq j \leq 2n+1}y_{(i+1)j}(j-(2n-h+1))$. Therefore,
\begin{equation*}
	\beta^h_i=-\alpha_{i+1,h}\dfrac{d}{dz}\mathlarger{\mathlarger{\Biggl(}}\dfrac{\mathlarger{\prod}_{1 \leq m \leq 2n+1, m \neq i+1} (1-x_mz)}{z^{h}\mathlarger{\prod}_{1 \leq m \leq 2n+1, m \neq i+1}(x_m-x_{i+1})}\mathlarger{\mathlarger{\Biggl)}}\mathlarger{\mathlarger{\Biggl|}}_{z=1}.
\end{equation*}

Since $x_{2n+1}=1$, we have that
\begin{equation*}
	\beta^h_i=\alpha_{i+1,h}^{-1}\mathlarger{\mathlarger{\Biggl(}}\dfrac{\mathlarger{\prod}_{1 \leq m \leq 2n, m \neq i+1} (1-x_m)}{\mathlarger{\prod}_{1 \leq m \leq 2n+1, m \neq i+1}(x_m-x_{i+1})}\mathlarger{\mathlarger{\Biggl)}}.
\end{equation*}

This finishes the proof of the proposition.
\end{proof}

Now, we can state \cite[Conjecture 3.17, Conjecture 3.26]{Cho2}.

\begin{conjecture}\label{conjecture3.17}
	(\cite[Conjecture 3.17, Conjecture 3.26]{Cho2}) For a basis $\lbrace x_1, \dots, x_{2n-m}, y_1, \dots, y_m \rbrace$ of $\BV$, and special cycles $\CZ(x_1),\dots,\CZ(x_{2n-m})$, and $\CY(y_1),\dots,\CY(y_m)$ in $\CN^n(1,2n-1)$, we have
	\begin{equation*}
		\begin{array}{l}
			\langle \CZ(x_1),\dots,\CZ(x_{2n-m}),\CY(y_1),\dots,\CY(y_m) \rangle\\:=\chi(O_{\CZ(x_1)}\otimes^{\BL}\dots\otimes^{\BL}O_{\CY(y_m)})\\
			=\dfrac{1}{W_{n,n}(A_n,0)}\lbrace W'_{m,n}(B,0)-\mathlarger{\sum}_{0 \leq i \leq n-1} \beta_i^mW_{m,i}(B,0)\rbrace.
			\end{array}
	\end{equation*}

Here $\chi$ is the Euler-Poincare characteristic and $\otimes^{\BL}$ is the derived tensor product. Also, $B$ is the matrix
	\begin{equation*}
	B=\left(\begin{array}{cc} 
		h(x_i,x_j) & h(x_i,y_l)\\
		h(y_k,x_j)& h(y_k,y_l)
	\end{array} 
	\right)_{1\leq i,j \leq 2n-m, 1\leq k,l \leq m}.
\end{equation*}
\end{conjecture}

\section{Reformulations of conjectures on $\CN^0(1,n-1)$}\label{section3}
In this section, we will write the derivative of a certain representation density as a linear sum of representation densities. This can be regarded as a version of \cite[Corollary 3.16]{CY} (orthogonal case) and its unitary variant \cite[Theorem 3.5.1]{LZ}.

\subsection{Linear independence}\label{subsection3.1}

In this subsection, let $K_n=GL_n(O_E)$. Then, we know that the complete set of representatives of $K_n \backslash X_n(E)$ is given by the set of diagonal matrices $\diag(\pi^{\lambda_1},\dots,\pi^{\lambda_n})$ where $\lambda_1 \geq \dots \geq \lambda_n$, and $\lambda_i \in \BZ$.

Let $\CR_{n}$ be the set
\begin{equation*}
	\CR_{n}=\lbrace Y_{\sigma,e} \vert (\sigma,e) \in \CS_{n}\times \BZ^{n}, \sigma^2=1, e_i=e_{\sigma(i)} \forall i \rbrace,
\end{equation*}
where $\CS_{n}$ is the symmetric group of degree $n$. For $\sigma \in \CS_{n}$, let

\begin{displaymath}
	Y_{\sigma,e}=\sigma \left(\begin{array}{ccc} 
		\pi^{e_1} &  & 0\\
		& \ddots & \\
		0 &   &\pi^{e_{n}}
	\end{array} 
	\right).
\end{displaymath}

We denote by $\CR^0_n$ the set
\begin{equation*}
	\CR^0_n:=\lbrace A_{\lambda} \vert \lambda=(\lambda_1,\dots, \lambda_n), \lambda_i \in \BZ, \lambda_1 \geq \dots \geq \lambda_n \rbrace,
\end{equation*}
where $A_{\lambda}=\diag(\pi^{\lambda_1},\dots,\pi^{\lambda_n})$. Sometimes, we will abuse notation slightly by regarding $\CR^0_n$ as the set of $\lambda$'s.

Also, we write $\CR_n^{0+}$ for the set
\begin{equation*}
	\CR^{0+}_n:=\lbrace A_{\lambda} \vert \lambda=(\lambda_1,\dots, \lambda_n), \lambda_i \in \BZ, \lambda_1 \geq \dots \geq \lambda_n \geq 0 \rbrace,
\end{equation*}
and $\CR_n^{0k}$ for the set
\begin{equation*}
	\CR^{0k}_n:=\lbrace A_{\lambda} \vert \lambda=(\lambda_1,\dots, \lambda_n), \lambda_i \in \BZ, k \geq \lambda_1 \geq \dots \geq \lambda_n \geq 0 \rbrace.
\end{equation*}

Let us recall the definition of usual representation densities. For $A \in X_m(O_E)$ and $B \in X_n(O_E)$ we define $\alpha(A,B)$ by
\begin{equation*}
	\alpha(A,B):=\mathlarger{\lim}_{d\rightarrow \infty} (q^{-d})^{n(2m-n)}\vert \FA_d(A,B) \vert,
\end{equation*}
where $\FA_d(A,B)=\lbrace x \in M_{m,n}(O_E/\pi^dO_E) \vert A[x]\equiv B (\Mod \pi^d)\rbrace$.

For $A \in X_n(O_E)$ and $B \in X_n(O_E)$ and $X=(-q)^{-2r}$, where $m=n+2r$, we define
\begin{equation*}
	\alpha(A,B;X):=\mathlarger{\lim}_{d\rightarrow \infty} (q^{-d})^{n(2m-n)}\vert \FA_d(A^{[r]},B) \vert,
\end{equation*}
where
\begin{equation*}
	A^{[r]}=\left(\begin{array}{cc}
		A & \\
		 & 1_{2r} \end{array}\right).
\end{equation*}

Also, we define $\alpha'(A,B)$ as
\begin{equation*}
	\alpha'(A,B)=-\dfrac{d}{dX}\alpha(A,B;X)\vert_{X=1}.
\end{equation*}
Recall the following formula for these representation densities in \cite{Hir}.

\begin{proposition}\label{proposition3.1}
	For $A \in X_n(O_E)$ and $B \in X_n(O_E)$, we have
	\begin{equation*}
		\alpha(A,B)=\mathlarger{\sum}_{Y \in \Gamma_{n} \backslash X_{n}(E)} \dfrac{\CG(Y,B)\CF_0(Y,A)}{\alpha(Y;\Gamma_{n})}.
	\end{equation*}
Here we use the notation in Lemma \ref{lemma2.2} and 
	\begin{equation*}
	\CF_0(Y,A)=\int_{M_{n,n}(O_E)} \psi(\langle Y, A[X]\rangle)dX.
\end{equation*}
\end{proposition}

As in the proof of \cite[Lemma 3.15]{Cho2} we can compute that for $\lambda=(\lambda_1,\dots,\lambda_n) \in \CR_n^{0+}$ and
\begin{displaymath}
	Y=\sigma \left(\begin{array}{ccc} 
		\pi^{e_1} &  & 0\\
		& \ddots & \\
		0 &   &\pi^{e_{n}}
	\end{array} 
	\right),
\end{displaymath}
\begin{equation*}
	\CF_0(Y,A_{\lambda})=\mathlarger{\prod}_{1 \leq i,j\leq n} (-q)^{\min(0,e_i+\lambda_j)}.
\end{equation*}

Now let us introduce some notation. Let \begin{equation*}
	f(Y)=\mathlarger{\prod}_i(-q)^{n\min(0,e_j)},
\end{equation*} 

and for an integer $k \geq 0$, 
\begin{equation*}
\CB_{k}(Y):=\mathlarger{\sum}_{i}\min(0,e_i+k)-\min(0,e_i).
\end{equation*} 

Then, we have
\begin{equation}\label{equation3.1.1}
	\CF_0(Y,A_{\lambda})=(-q)^{\sum_{i}\CB_{\lambda_i}(Y)}f(Y).
\end{equation}

Note that
\begin{equation*}
	\min(0,e+\lambda)-\min(0,e)=\left\lbrace \begin{array}{ll}
		\lambda & \text{ if } e \leq -\lambda;\\
		\lambda-1 & \text{ if } e =-(\lambda-1);\\
		\vdots& \\
		0 & \text{ if } e \geq 0.
		\end{array}\right.
\end{equation*}

From this, we can see that $\CF_0(Y,A_{\lambda})$ depends only on the tuples $(\# e_i \geq 0, \# e_i=-1, \# e_i=-2, \dots)$. Therefore, let us define

\begin{equation*}
	E_0(Y):=\# \lbrace e_i \vert e_i \geq 0 \rbrace,
\end{equation*}
and for $k \geq 1$,
\begin{equation*}
	E_k(Y):=\# \lbrace e_i \vert e_i=-k \rbrace.
\end{equation*}

Similarly, for $\eta \in \CR_n^{0+}$ and $k \geq 0$, we define
\begin{equation*}
	E_k(\eta):=\# \lbrace \eta_i \vert \eta_i=k \rbrace,
\end{equation*}

For each $\eta \in \CR_n^{0+}$, we define the subset $\FC_{\eta}$ of $\CR_n$ by

\begin{equation*}
	\FC_{\eta}:=\left\lbrace \begin{array}{l}Y \in \CR_n \end{array}\middle\vert \begin{array}{c}E_0(Y)=E_0(\eta)\\
		E_1(Y)=E_1(\eta)\\
		\vdots\\
	\end{array}\right\rbrace.
\end{equation*}

For $\alpha \in \CR_m^{0+}$ and $\eta \in \CR_n^{0+}$, we write $\CB_{\alpha}(\eta)$ for
\begin{equation*}
	\CB_{\alpha}(\eta)=\mathlarger{\sum}_{1 \leq i \leq m, 1 \leq j \leq n} \min(\alpha_i,\eta_j).
\end{equation*}
Then for $\alpha \in \CR_n^{0+}$, $\eta \in \CR_n^{0+}$, and $Y \in \FC_{\eta}$, we have
\begin{equation}\label{equation3.1.2}
	\CF_0(Y,A_{\alpha})=(-q)^{\CB_{\alpha}(\eta)}f(Y).
\end{equation}

Now, since $\CF_0(Y,A_{\lambda})$ depends only on $\eta \in \CR_n^{0+}$ such that $Y \in \FC_{\eta}$, we can see that $\CF_0(\cdot,A_{\lambda})$ as a function on $\CR_n^{0+}$. For $\lambda \in \CR_n^{0+}$, we can consider $\CF_0(\cdot,A_{\lambda})$ as a function on $\CR_n^{0+}$. Then, we can state the following lemma.

\begin{lemma}\label{lemma3.2}
	For any $k \geq 0$ and $\lambda_1, \dots, \lambda_m \in \CR_n^{0k}$ such that none of them are the same, the restrictions of functions $\CF_0(\cdot, A_{\lambda_i})$ to $\CR_n^{0k}$ are linearly independent.
\end{lemma}
\begin{proof}
	Assume that there are $\alpha_i \in \CR_n^{0k}$ and constants $c_{\alpha_i}$ such that
	\begin{equation*}
	 \mathlarger{\sum}_{i}c_{\alpha_i} \CF_0(Y, A_{\alpha_i})=0,\quad \forall Y \in \FC_{\eta}, \forall \eta \in \CR_n^{0k}.
	 \end{equation*}
 Then, we claim that
 \begin{equation*}
 	\mathlarger{\sum}_i c_{\alpha_i} \alpha(A_{\alpha_i},B)=0, \quad \forall B \in \CR_n^{0(k-1)}.
 \end{equation*}

By proposition \ref{proposition3.1}, we know that
\begin{equation*}
	\alpha(A_{\alpha_i},B)=\mathlarger{\sum}_{Y \in \Gamma_{n} \backslash X_{n}(E)} \dfrac{\CG(Y,B)\CF_0(Y,A_{\alpha_i})}{\alpha(Y;\Gamma_{n})}.
\end{equation*}

Also, in the proof of \cite[Lemma 3.8]{Cho2}, we know that for $B=\diag(\pi^{\lambda_1},\dots,\pi^{\lambda_{n}}) \in \CR_n^{0(k-1)}$ and
\begin{displaymath}
	Y=\sigma \left(\begin{array}{ccc} 
		\pi^{e_1} &  & 0\\
		& \ddots & \\
		0 &   &\pi^{e_{n}}
	\end{array} 
	\right),
\end{displaymath}
we have

\begin{equation*}
	\begin{array}{ll}
		\CG(Y,B)&=\mathlarger{\prod}_{k< j=\sigma(j)} \int_{O} \psi(\pi^{e_j+\lambda_k}\Nm(x))dx\\
		
		&\times \mathlarger{\prod}_{ k= j=\sigma(j)} \int_{O^{\times}} \psi(\pi^{e_j+\lambda_k}\Nm(x))dx\\
		&\times \mathlarger{\prod}_{k> j=\sigma(j)} \int_{\pi O} \psi(\pi^{e_j+\lambda_k}\Nm(x))dx\quad\quad\quad\quad\quad\quad\quad\quad\quad\quad
		
		\end{array}
	\end{equation*}
\begin{equation*}
	\begin{array}{ll}
\quad\quad\quad&\times \mathlarger{\prod}_{k<j, k<\sigma(j), j\neq \sigma(j)} (\int_{O \times O} \psi(\pi^{e_j+\lambda_k}\Tr(xy))dxdy)^{1/2}\\

		&\times \mathlarger{\prod}_{k=j, k<\sigma(j), j\neq \sigma(j)} (\int_{O^{\times} \times O} \psi(\pi^{e_j+\lambda_k}\Tr(xy))dxdy)^{1/2}\\
		
		&\times \mathlarger{\prod}_{k>j, k<\sigma(j), j\neq \sigma(j)} (\int_{\pi O \times O} \psi(\pi^{e_j+\lambda_k}\Tr(xy))dxdy)^{1/2}
		\end{array}
\end{equation*}
\begin{equation*}
	\begin{array}{ll}		
	\quad\quad\quad	&\times \mathlarger{\prod}_{k<j, k=\sigma(j), j\neq \sigma(j)} (\int_{O \times O^{\times}} \psi(\pi^{e_j+\lambda_k}\Tr(xy))dxdy)^{1/2}\\

		&\times \mathlarger{\prod}_{k>j, k=\sigma(j),j\neq \sigma(j)} (\int_{\pi O \times O^{\times}} \psi(\pi^{e_j+\lambda_k}\Tr(xy))dxdy)^{1/2}
				\end{array}
	\end{equation*}
\begin{equation*}
\begin{array}{ll}	
		
	\quad\quad\quad	&\times \mathlarger{\prod}_{k<j, k>\sigma(j), j\neq \sigma(j)} (\int_{O \times \pi O} \psi(\pi^{e_j+\lambda_k}\Tr(xy))dxdy)^{1/2}\\
		
		&\times \mathlarger{\prod}_{k=j, k>\sigma(j), j\neq \sigma(j)} (\int_{O^{\times} \times \pi O} \psi(\pi^{e_j+\lambda_k}\Tr(xy))dxdy)^{1/2}\\
		
		&\times \mathlarger{\prod}_{k>j, k>\sigma(j), j\neq \sigma(j)} (\int_{\pi O \times \pi O} \psi(\pi^{e_j+\lambda_k}\Tr(xy))dxdy)^{1/2}
		
	\end{array}
\end{equation*}

Note that
\begin{equation*}
	\int_{O^{\times}} \psi(\pi^e \Nm(x))dx=0, \quad \forall e<-1,
\end{equation*}
and
\begin{equation*}
	\int_O \psi(\pi^e \Tr(x))dx=0, \quad \forall e<0.
\end{equation*}
If $Y \notin \bigcup_{\eta \in \CR_n^{0(k)}} \FC_{\eta}$, then there is at least one $e_i$ such that $e_i+\lambda_k\leq -2$, $\forall k$, and hence $\CG(Y,B)=0$.

Therefore, in the formula 
\begin{equation*}
	\mathlarger{\sum}_i c_{\alpha_i}\alpha(A_{\alpha_i},B)=\mathlarger{\sum}_{Y \in \Gamma_{n} \backslash X_{n}(E)} \dfrac{\CG(Y,B)	\mathlarger{\sum}_i c_{\alpha_i}\CF_0(Y,A_{\alpha_i})}{\alpha(Y;\Gamma_{n})},
\end{equation*}

we have

\begin{equation*}
	\begin{array}{l}
			\mathlarger{\sum}_i c_{\alpha_i}\CF_0(Y,A_{\alpha_i})=0, \text{ if } Y \in \bigcup_{\eta \in \CR_n^{0(k)}} \FC_{\eta}\\
			\CG(Y,B)=0, \text{ if } Y \notin \bigcup_{\eta \in \CR_n^{0(k)}} \FC_{\eta}.
		\end{array}
\end{equation*}

This proves the above claim
 \begin{equation*}
	\mathlarger{\sum}_i c_{\alpha_i} \alpha(A_{\alpha_i},B)=0, \quad \forall B \in \CR_n^{0(k-1)}.
\end{equation*}

If there is $\alpha_t=(\alpha_{t1},\dots, \alpha_{tn})$ in $\CR_n^{0(k-1)}$ such that $c_{\alpha_t} \neq 0$, we can choose $\alpha_t$ that has the minimum sum $\alpha_{t1}+\dots+\alpha_{tn}$. For this $\alpha_t$, let $B=A_{\alpha_t}$. Then $\alpha(A_{\alpha_j},B)=0$ for all $j \neq t$ and $\alpha(A_{\alpha_t},B)\neq 0$. Since $\mathlarger{\sum}_i c_{\alpha_i} \alpha(A_{\alpha_i},B)=0$, we have that $c_{\alpha_t}=0$. This is a contradiction, and hence all $\alpha_i$ should be in $\CR_n^{0(k)}\backslash \CR_n^{0(k-1)}$.

Now, it suffices to show that $\CF_0(\cdot,A_{\alpha})$, $\alpha \in $ $\CR_n^{0(k)}\backslash \CR_n^{0(k-1)}$ are linearly independent on $\CR_n^{0(k)}\backslash \CR_n^{0(k-1)}$. Note that $\alpha \in \CR_n^{0(k)}\backslash \CR_n^{0(k-1)}$ can be written as $(k,\overline{\alpha})$ where $\overline{\alpha}$ is in $\CR_{n-1}^{0k}$.

We will show that the linear independence of $\CF_0(\cdot,A_{\alpha})$, $\alpha \in $ $\CR_n^{0(k)}\backslash \CR_n^{0(k-1)}$ as functions on $\CR_n^{0(k)}\backslash \CR_n^{0(k-1)}$ comes from the linear independence of $\CF_0(\cdot,A_{\overline{\alpha}})$, $\overline{\alpha} \in \CR_{n-1}^{0k}$. Then we can inductively show that the linear independence of $\CF_0(\cdot,A_{\alpha})$,  $\alpha \in $ $\CR_n^{0(k)}$ comes from the linear independence of $\CF_0(\cdot,A_{\alpha})$,  $\alpha \in $ $\CR_1^{0(k)}$ that is trivial.

From \eqref{equation3.1.1}, we have

\begin{equation*}
	\CF_0(Y,A_{\lambda})=(-q)^{\sum_{i}\CB_{\lambda_i}(Y)}f(Y),
\end{equation*}
where
\begin{equation*}
	\CB_{\lambda}(Y):=\mathlarger{\sum}_{i}\min(0,e_i+\lambda)-\min(0,e_i).
\end{equation*} 

Also, from \eqref{equation3.1.2}, we have that for $\alpha \in \CR_n^{0+}$, $\eta \in \CR_n^{0+}$, and $Y \in \FC_{\eta}$,
\begin{equation*}
	\CF_0(Y,A_{\alpha})=(-q)^{\CB_{\alpha}(\eta)}f(Y),
\end{equation*}
where

\begin{equation*}
	\CB_{\alpha}(\eta)=\mathlarger{\sum}_{1 \leq i \leq m, 1 \leq j \leq n} \min(\alpha_i,\eta_j).
\end{equation*}

For $\alpha, \eta \in \CR_n^{0(k)}\backslash \CR_n^{0(k-1)}$, we write $\overline{\alpha}$ and $\overline{\beta}$ for the elements in $\CR_{n-1}^{0k}$ such that
\begin{equation*}
	\begin{array}{l}
		\alpha=(k,\overline{\alpha})\\
		\beta=(k,\overline{\beta}).
		\end{array}
\end{equation*}
Then for $Y_{\eta} \in \FC_{\eta}$ and $\alpha$, we have
\begin{equation*}\begin{array}{l}
	\CF_0(Y_{\eta},A_{\alpha})=(-q)^{\CB_{\overline{\alpha}}(\overline{\eta})+\CB_{k}(k)+\CB_{k}(\overline{\eta})+\CB_{\overline{\alpha}}(k)}\\
	=\CF_0(Y_{\overline{\eta}},A_{\overline{\alpha}}) \times (-q)^{k+\mathlarger{\sum}_{1 \leq i \leq n-1} \overline{\eta}_i + \mathlarger{\sum}_{1 \leq i \leq n-1} \overline{\alpha}_i}\\
	
	\end{array}
\end{equation*}

Now, let us give the lexicographical order on the finite sets $\CR_n^{0k}\backslash \CR_n^{0(k-1)}$ and $\CR_{n-1}^{0k}$ (note that these are bijective) and regard $(\CF_0(Y_{\eta},A_{\alpha}))_{\eta, \alpha}$ as a matrix with respect to this order. Also, let $X$ be the matrix $(\diag((-q)^{\sum_i \overline{\alpha}_i}))_{\overline{\alpha}}.$

Then, we have that
\begin{equation*}
(\CF_0(Y_{\eta},A_{\alpha}))_{\eta, \alpha}=(-q)^k (^tX) (\CF_0(Y_{\overline{\eta}},A_{\overline{\alpha}}))_{\overline{\eta}, \overline{\alpha}}X.
\end{equation*}

Since $X$ is an invertible matrix, the matrix $(\CF_0(Y_{\eta},A_{\alpha}))_{\eta, \alpha}$, $\alpha,\eta \in \CR_n^{0k}\backslash \CR_n^{0(k-1)}$ is invertible if and only if the matrix $(\CF_0(Y_{\overline{\eta}},A_{\overline{\alpha}}))_{\overline{\eta},\overline{\alpha}}$, $\overline{\alpha}, \overline{\eta} \in \CR_{n-1}^{0k}$ is invertible.

Therefore, the linear independence of $\CF_0(\cdot,A_{\alpha})$, $\alpha \in $ $\CR_n^{0(k)}\backslash \CR_n^{0(k-1)}$ as functions on $\CR_n^{0(k)}\backslash \CR_n^{0(k-1)}$ comes from the linear independence of $\CF_0(\cdot,A_{\overline{\alpha}})$, $\overline{\alpha} \in \CR_{n-1}^{0k}$.

This finishes the proof of the lemma.

\end{proof}

\subsection{Reformulations of conjectures on $\CN^0(1,n-1)$}\label{subsection3.2}

In this subsection, we will write $\dfrac{\alpha'(1_n,B)}{\alpha(1_n,1_n)}$ as a sum $\mathlarger{\sum}_{\alpha \in \CR_n^{0+}}C_{\alpha} \dfrac{\alpha(A_{\alpha},B)}{\alpha(A_{\alpha},A_{\alpha})}$ for some constants $C_{\alpha}$.

Assume that we have a hermitian $O_E$-lattice $L=O_Ex_1 \oplus \dots O_Ex_n$ of rank $n$ with hermitian form $B$. For a fixed $\alpha \in \CR_n^{0+}$, let $X$ be a nonsingular element in $M_{n,n}(O_E)$ such that
\begin{equation*}
A_{\alpha}[X]=B.
\end{equation*}

Then, one can consider the $O_E$-lattice $LX^{-1}$. Since $A_{\alpha}=(^tX^*)^{-1}BX^{-1}$ the lattice $LX^{-1}$ has the hermitian form $A_{\alpha}$. Note that this lattice contains $L$.

Therefore, for a fixed lattice $L$ with hermitian form $B$, one can regard $\dfrac{\alpha(A_{\alpha},B)}{\alpha(A_{\alpha},A_{\alpha})}$ as counting of different lattices $L_{\alpha}$ (since $\alpha(A_{\alpha},A_{\alpha})$ gives the number of automorphisms) with hermitian form $A_{\alpha}$ containing $L$.

In this way, we can reformulate the unitary version of Cho-Yamauchi's formula \cite[Theorem 3.5.1]{LZ} as a linear sum
\begin{equation*}
	\mathlarger{\sum}_{\alpha \in \CR_n^{0+}} C_{\alpha} \dfrac{\alpha(A_{\alpha},B)}{\alpha(A_{\alpha},A_{\alpha})}.
\end{equation*}

We recall from Lemma \ref{lemma2.2} that
	\begin{equation*}
	W_{0,0}(B,r)=\mathlarger{\sum}_{Y \in \Gamma_{2n} \backslash X_{2n}(E)} \dfrac{\CG(Y,B)\CF_0(Y,A_0^{[r]})}{\alpha(Y;\Gamma_{2n})},
\end{equation*}
and
\begin{equation}\label{equation3.2.1}
	\alpha'(1_n,B)=W_{0,0}'(B,0)=\mathlarger{\sum}_{Y \in \Gamma_{2n} \backslash X_{2n}(E)} \dfrac{\CG(Y,B)\CF_0'(Y,A_0^{[0]})}{\alpha(Y;\Gamma_{2n})}.
\end{equation}

Note that $\CF_0'(Y,A_0^{[0]})$ depends only on $\eta \in \CR_n^{0+}$ such that $Y \in \FC_{\eta}$.

By Lemma \ref{lemma3.2}, we know that $\CF_0(\cdot,A_{\alpha})$, $\alpha \in \CR_n^{0k}$ are linearly independent as functions on $\CR_n^{0k}$ and the number of functions $\CF_0(\cdot,A_{\alpha})$ is $\vert \CR_n^{0k} \vert$. This means that any function on $\CR_n^{0k}$ can be written uniquely as a linear sum of $\CF_0(\cdot,A_{\alpha})$, $\alpha \in \CR_n^{0k}$. Furthermore, any convergent function on $\CR_n^{0+}$ can be written uniquely as a linear sum of $\CF_0(\cdot, A_{\alpha})$, $\alpha \in \CR_n^{0+}$.

Therefore, $\CF_0'(Y,A_0^{[0]})$ can be written uniquely as a linear sum of $\CF_0(\cdot, A_{\alpha})$, $\alpha \in \CR_n^{0+}$.

Assume that
\begin{equation}\label{equation3.2.2}
\CF_0'(Y,A_0^{[0]})=\mathlarger{\sum}_{\alpha} c_{\alpha} \CF_0(Y,A_{\alpha}).
\end{equation}

Then by Proposition \ref{proposition3.1}, \eqref{equation3.2.1}, and \eqref{equation3.2.2}, we have
\begin{equation*}
		\alpha'(1_n,B)=\mathlarger{\sum}_{\alpha} c_{\alpha} \alpha(A_{\alpha},B).
\end{equation*}

Therefore, there are unique constants $C_{\alpha}$, $\alpha \in \CR_n^{0+}$ such that
\begin{equation*}
	\dfrac{\alpha'(1_n,B)}{\alpha(1_n,1_n)}=\mathlarger{\sum}_{\alpha} C_{\alpha} \dfrac{\alpha(A_{\alpha},B)}{\alpha(A_{\alpha},A_{\alpha})}.
\end{equation*}

For $\alpha \in \CR_n^{0+}$ such that $\det(A_{\alpha})$ has an odd valuation, $C_{\alpha}$ was computed in \cite[Corollary 3.5.3]{LZ} as follows.
\begin{proposition}(\cite[Corollary 3.5.3]{LZ})\label{proposition3.3}
	For $\alpha=(\alpha_1, \dots, \alpha_n) \in \CR_n^{0+}$, we write $t(\alpha)$ for the number of nonzero $\alpha_i$'s.
	
	Then for $\alpha$ such that $\sum_i \alpha_i$ is odd, we have
	\begin{equation*}
		C_{\alpha}=\mathlarger{\prod}_{i=1}^{t(\alpha)-1}(1-(-q)^i).
	\end{equation*}

Here, $C_{\alpha}=1$ if $t(\alpha)=1$.
\end{proposition}

Now, we need to find $C_{\alpha}$ for $\alpha$ with even $\sum_i \alpha_i$.

\begin{proposition}\label{proposition3.4}
		For $\alpha=(\alpha_1, \dots, \alpha_n) \in \CR_n^{0+}$, we write $t(\alpha)$ for the number of nonzero $\alpha_i$'s.
	
	Then for $\alpha \neq (0,0,\dots, 0)$ such that $\sum_i \alpha_i$ is even, we have
	\begin{equation*}
		C_{\alpha}=-\mathlarger{\prod}_{i=1}^{t(\alpha)-1}(1-(-q)^i),
	\end{equation*}
and for $\alpha=(0,0,\dots,0)$, we have
\begin{equation*}
	C_{(0,0,\dots,0)}=\dfrac{\alpha'(1_n,1_n)}{\alpha(1_n,1_n)}.
\end{equation*}
Here, $C_{\alpha}=-1$ if $t(\alpha)=1$
\end{proposition}
\begin{proof}
	In the proof of this proposition, we will use \cite[Theorem 3.5.1]{LZ} and the functional equation \cite[(3.2.0.2)]{LZ} (\cite[Theorem 5.3]{Hir3}). Let us write these in our notation. Let $L$ be a lattice whose hermitian form is $B$ and $\det B$ is even. Then, \cite[Theorem 3.5.1]{LZ} can be written as
	\begin{equation}\label{equation3.2.3.4.1}
		\dfrac{\alpha(1_n,B;X)}{\alpha(1_n,1_n;X)}=\mathlarger{\sum}_{L \subset L' \subset L'^{\vee}} X^{2l(L'/L)}m(t(L');X).
	\end{equation}
where $l(L'/L)=\text{length}_{O_E}L'/L$,
\begin{equation*}
	m(a;X):=\mathlarger{\prod}_{i=0}^{a-1}(1-(-q)^iX), \forall \text{ } a \geq 1,
\end{equation*}
and $m(0;X):=1$.
Also, we define $t(L')$ as the number of nonzero $a_i$'s and $\val(L')=\sum a_i$, where $L'^{\vee}/L'\simeq O_E/\pi^{a_i}$ as $O_E$-modules.

We used the notation $m(a)$ for
\begin{equation*}
	m(a)=-\dfrac{d}{dX} m(a;X)=\mathlarger{\prod}_{i=1}^{a-1} (1-(-q)^i), \forall \text{ } a \geq 2,
\end{equation*}
and $m(0):=0, m(1):=1$.

The functional equation \cite[(3.2.0.2)]{LZ} can be written as
\begin{equation}\label{equation3.2.3.4.2}
	\dfrac{\alpha(1_n,B;X)}{\alpha(1_n,1_n;X)}=(-X)^{\val (L)} \dfrac{\alpha(1_n,B;1/X)}{\alpha(1_n,1_n;1/X)}
\end{equation}

From this, we have that
\begin{equation*}
	-\dfrac{d}{dX}\lbrace\dfrac{\alpha(1_n,B;X)}{\alpha(1_n,1_n;X)}\rbrace\vert_{X=1}=-\dfrac{d}{dX}\lbrace(-X)^{\val (L)} \dfrac{\alpha(1_n,B;1/X)}{\alpha(1_n,1_n;1/X)}\rbrace\vert_{X=1}
\end{equation*}
and hence, we have
\begin{equation*}\begin{array}{ll}
	\dfrac{\alpha'(1_n,B)}{\alpha(1_n,1_n)}-\dfrac{\alpha(1_n,B)}{\alpha(1_n,1_n)^2}\alpha'(1_n,1_n)&=(-1)^{\val(L)+1}\val(L)\dfrac{\alpha(1_n,B)}{\alpha(1_n,1_n)}\\&+(-1)^{\val(L)}\dfrac{-\alpha'(1_n,B)}{\alpha(1_n,1_n)}\\&
	+(-1)^{\val(L)}\dfrac{\alpha(1_n,B)}{\alpha(1_n,1_n)^2}\alpha'(1_n,1_n).
\end{array}	
\end{equation*}

Since $\val(L)$ is even, we have that
\begin{equation}\label{equation3.2.3.4.3}
		2\dfrac{\alpha'(1_n,B)}{\alpha(1_n,1_n)}-2\dfrac{\alpha(1_n,B)}{\alpha(1_n,1_n)^2}\alpha'(1_n,1_n)=-\val(L)\dfrac{\alpha(1_n,B)}{\alpha(1_n,1_n)}.
\end{equation}

On the other hand, from \eqref{equation3.2.3.4.1}, we have
\begin{equation*}\begin{array}{l}
	\dfrac{\alpha'(1_n,B)}{\alpha(1_n,1_n)}-\dfrac{\alpha(1_n,B)}{\alpha(1_n,1_n)^2}\alpha'(1_n,1_n)=-\dfrac{d}{dX}\mathlarger{\sum}_{L \subset L' \subset L'^{\vee}} X^{2l(L'/L)}m(t(L');X)\vert_{X=1}.
	\end{array}
\end{equation*}

Note that $-\dfrac{d}{dX} m(t(L');X)\vert_{X=1}=m(t(L'))$ and $m(t(L'),1)=0$, if $t(L')\geq 1$. Therefore, we have

\begin{equation*}\begin{array}{ll}
		\dfrac{\alpha'(1_n,B)}{\alpha(1_n,1_n)}-\dfrac{\alpha(1_n,B)}{\alpha(1_n,1_n)^2}\alpha'(1_n,1_n)&=\mathlarger{\sum}_{\substack{L \subset L' \subset L'^{\vee}\\ t(L') \neq 0}} m(t(L'))\\
		&+\mathlarger{\sum}_{\substack{L \subset L' \subset L'^{\vee}\\ t(L')=0}}
		-2l(L'/L).
	\end{array}
\end{equation*}

Note that $t(L')=0$ means that $L'$ has a hermitian matrix $1_n$ and \begin{equation*}
	2l(L'/L)=\val(L).
	\end{equation*}
Therefore, we have
\begin{equation*}\begin{array}{ll}
		\dfrac{\alpha'(1_n,B)}{\alpha(1_n,1_n)}-\dfrac{\alpha(1_n,B)}{\alpha(1_n,1_n)^2}\alpha'(1_n,1_n)&=\mathlarger{\sum}_{\substack{L \subset L' \subset L'^{\vee}\\ t(L') \neq 0}} m(t(L'))\\
		&+\mathlarger{\sum}_{\substack{L \subset L' \subset L'^{\vee}\\ t(L')=0}}
		-\val(L)\\
		&=\mathlarger{\sum}_{\substack{L \subset L' \subset L'^{\vee}\\ t(L') \neq 0}} m(t(L'))\\
		&-\val(L)\dfrac{\alpha(1_n,B)}{\alpha(1_n,1_n)}.
	\end{array}
\end{equation*}

Combining this with \eqref{equation3.2.3.4.3}, we have
\begin{equation*}\begin{array}{ll}
		\dfrac{\alpha'(1_n,B)}{\alpha(1_n,1_n)}-\dfrac{\alpha(1_n,B)}{\alpha(1_n,1_n)^2}\alpha'(1_n,1_n)&=-\mathlarger{\sum}_{\substack{L \subset L' \subset L'^{\vee}\\ t(L') \neq 0}} m(t(L'))\\
	\end{array}
\end{equation*}
and hence
\begin{equation*}\begin{array}{ll}
		\dfrac{\alpha'(1_n,B)}{\alpha(1_n,1_n)}&=\dfrac{\alpha(1_n,B)}{\alpha(1_n,1_n)^2}\alpha'(1_n,1_n)-\mathlarger{\sum}_{\substack{L \subset L' \subset L'^{\vee}\\ t(L') \neq 0}} m(t(L'))\\
		&=\dfrac{\alpha(1_n,B)}{\alpha(1_n,1_n)^2}\alpha'(1_n,1_n)-\mathlarger{\sum}_{\substack{\alpha \\ \sum \alpha_i=\text{even} \\t(\alpha)\neq 0}} m(t(\alpha))\dfrac{\alpha(A_{\alpha},B)}{\alpha(A_{\alpha},A_{\alpha})}.
	\end{array}
\end{equation*}

This proves the proposition.
\end{proof}
\begin{remark}
 If $\det B$ has an odd valuation, $\alpha(A_{\alpha},B)=0$ for $\alpha$'s with even $\sum_i \alpha_i$. Therefore, we may not need $C_{\alpha}$ for these $\alpha$'s in the case of hyperspecial level structure $\CN^0(1,n-1)$ since the determinant of the hermitian matrix $B$ of special homomorphisms has an odd valuation. But, in general in $\CN^n(1,2n-1)$, $\det B$ can have an even valuation.
	
\end{remark}

\section{Reformulations of conjectures on $\CN^n(1,2n-1)$}\label{section4}
In this section, we consider the arithmetic intersection numbers of $\CZ(x_1)$, $\CZ(x_2),$$\dots$, $\CZ(x_{2n})$ in $\CN^n(1,2n-1)$. In Conjecture \ref{conjecture3.17}, we conjecture that
	\begin{equation*}
	\begin{array}{l}
		\langle \CZ(x_1),\dots,\CZ(x_{2n}) \rangle
		=\dfrac{1}{W_{n,n}(A_n,0)}\lbrace W'_{0,n}(B,0)-\mathlarger{\sum}_{0 \leq i \leq n-1} \beta_i^0W_{0,i}(B,0)\rbrace.
	\end{array}
\end{equation*}
We will express this formula in terms of weighted counting of lattices containing $x_1, \dots, x_{2n}$.

\subsection{A Reformulation of $W'_{0,n}(B,0)$}\label{subsection4.1}
In this subsection, we will write $\dfrac{W'_{0,n}(B,0)}{W_{n,n}(A_n,0)}$ as a sum $\mathlarger{\sum}_{\alpha \in \CR_n^{0+}}D_{\alpha} \dfrac{\alpha(A_{\alpha},B)}{\alpha(A_{\alpha},A_{\alpha})}$ for some constants $D_{\alpha}$. Recall from Lemma \ref{lemma2.2} and Proposition \ref{proposition3.1} that
for $B \in X_{2n}(E)$ and $\lambda \in \CR_{2n}^{0+}$, we have
\begin{equation*}
	W_{0,n}'(B,0)=\mathlarger{\sum}_{Y \in \Gamma_{2n} \backslash X_{2n}(E)} \dfrac{\CG(Y,B)\CF_0'(Y,A_n^{[0]})}{\alpha(Y;\Gamma_{2n})},
\end{equation*}
and
	\begin{equation*}
	\alpha(A_{\lambda},B)=\mathlarger{\sum}_{Y \in \Gamma_{2n} \backslash X_{2n}(E)} \dfrac{\CG(Y,B)\CF_0(Y,A_{\lambda})}{\alpha(Y;\Gamma_{2n})}.
\end{equation*}

As in Section \ref{subsection3.2}, we need to express $\CF_0'(Y,A_n^{[0]})$ as a linear sum of $\CF_0(Y,A_{\lambda})$, for $\lambda \in \CR^{0+}_{2n}$. To find this expression, we will use Proposition \ref{proposition3.3} and Proposition \ref{proposition3.4}.

As in the proof of \cite[Lemma 3.15]{Cho2} we can compute that for \begin{displaymath}
	Y=\sigma \left(\begin{array}{ccc} 
		\pi^{e_1} &  & 0\\
		& \ddots & \\
		0 &   &\pi^{e_{2n}}
	\end{array} 
	\right),
\end{displaymath}
\begin{equation*}
	\CF_0'(Y,A_n^{[0]})=(\mathlarger{\sum}_{j}\min(0,e_j))(-q)^{-4n^2}(-q)^{n\CB_1(Y)}f(Y),
\end{equation*}
and
\begin{equation*}
	\CF_0'(Y,A_0^{[0]})=(\mathlarger{\sum}_{j}\min(0,e_j))f(Y).
\end{equation*}

Note that if we consider $\CF_0'(Y,A_n^{[0]})-(-q)^{-2n^2}\CF_0'(Y,A_0^{[0]})$, we have
\begin{equation}\label{eq.4.1.0}
	\begin{array}{l}
	\CF_0'(Y,A_n^{[0]})-\dfrac{\CF_0'(Y,A_0^{[0]})}{(-q)^{2n^2}}\\
	=\left\lbrace \begin{array}{cl} 
		0 & \text{if } Y \in \FC_{\eta},\\
		&E_0(\eta)=0\\\\
	 (\mathlarger{\sum}_{j}\min(0,e_j))f(Y)((-q)^{n(2n-1)-4n^2}-(-q)^{-2n^2})& \text{ if } Y \in \FC_{\eta},
	 \\& E_0(\eta)=1\\
	 \quad\quad\quad\quad\vdots\\
	 (\mathlarger{\sum}_{j}\min(0,e_j))f(Y)((-q)^{n(2n-k)-4n^2}-(-q)^{-2n^2})& \text{ if } Y \in \FC_{\eta},\\
	 & E_0(\eta)=k\\
	 \quad\quad\quad\quad\vdots\\
	 (\mathlarger{\sum}_{j}\min(0,e_j))f(Y)((-q)^{-4n^2}-(-q)^{-2n^2})& \text{ if } Y \in \FC_{\eta}, \\
	 &E_0(\eta)=2n

		\end{array}\right.
		\end{array}
\end{equation}

Since this depends only on $E_0(\eta)$, let us define
\begin{equation*}
	\FD_k:=\lbrace Y \in \FC_{\eta}\vert \eta \in \CR_{2n}^{0+}, E_0(\eta)=k\rbrace.
	\end{equation*}
Also, we will abuse notation slightly by regarding $\FD_k$ as the set of
elements $\eta$ in $\CR_{2n}^{0+}$ such that  $E_0(\eta)=k$.

Now, let us consider $n=1$ case. We have the following proposition.
\begin{proposition}\label{proposition4.1}
	For $\lambda=(\overline{\lambda},0) \in \FD_1 \cup \FD_2 \subset \CR_2^{0+}$, we define
	\begin{equation*}
		\lambda_1^+:=(\overline{\lambda},1) \in \CR_2^{0+}.
	\end{equation*}
	Here, for $\lambda=(0,0)$, we assume $\lambda_1^+=(1,0) \in \CR_2^{0+}$ after changing the order. Then, we have
	\begin{equation*}
		\begin{array}{l}
	\CF_0'(Y,A_1^{[0]})-\dfrac{\CF_0'(Y,A_0^{[0]})}{(-q)^{2}}\\\\
	
	=\mathlarger{\sum}_{\lambda \in \FD_1 \cup \FD_2} C_{\overline{\lambda}}\dfrac{\alpha(1_1,1_1)}{\alpha(A_{\overline{\lambda}},A_{\overline{\lambda}})}(-q)^{-4}\lbrace \CF_0(Y,A_{\lambda_1^+})-(-q)^2\CF_0(Y,A_{\lambda})\rbrace.
	\end{array}\end{equation*}
\end{proposition}

\begin{proof}
Note that for
\begin{equation*}
	Y=\sigma\left( \begin{array}{cc} \pi^{e_1} & \\
		& \pi^{e_2}\end{array}\right),
\end{equation*}
\begin{equation}\label{eq4.1.1}
	\begin{array}{l}
		\CF_0'(Y,A_1^{[0]})-\dfrac{\CF_0'(Y,A_0^{[0]})}{(-q)^{2}}\\
		=\left\lbrace \begin{array}{cl} 
			0 & \text{if } Y \in \FD_0\\
			(\mathlarger{\sum}_{j}\min(0,e_j))f(Y)((-q)^{-3}-(-q)^{-2})& \text{ if } Y \in \FD_1\\
			(\mathlarger{\sum}_{j}\min(0,e_j))f(Y)((-q)^{-4}-(-q)^{-2})& \text{ if } Y \in \FD_2\\
				\end{array}\right.
	\end{array}
\end{equation}
and for $\lambda \in \FD_1 \cup \FD_2$, we have
\begin{equation}\label{eq4.1.2}
\begin{array}{l}
	\CF_0(Y,A_{\lambda_1^+})-(-q)^2\CF_0(Y,A_{\lambda})\\
	=\left\lbrace \begin{array}{cl} 
		0 & \text{if } Y \in \FD_0\\
		f(Y)(-q)^{\CB_{\overline{\lambda}}(Y)}((-q)-(-q)^{2})& \text{ if } Y \in \FD_1\\
		f(Y)(-q)^{\CB_{\overline{\lambda}}(Y)}(1-(-q)^{2})& \text{ if } Y \in \FD_2\\
	\end{array}\right.
\end{array}
\end{equation}

Now, recall from Proposition \ref{proposition3.3} and Proposition \ref{proposition3.4} that for
\begin{equation*}
	Y=(\pi^{e})
\end{equation*}
we have
\begin{equation*}
	\CF_0'(Y,1_1)=\mathlarger{\sum}_{\overline{\lambda} \in \CR_1^{0+}}C_{\overline{\lambda}}\dfrac{\alpha(1_1,1_1)}{\alpha(A_{\overline{\lambda}},A_{\overline{\lambda}})}\CF_0(Y,A_{\overline{\lambda}})
\end{equation*}

This implies that
\begin{equation}\label{eq4.1.3}
	\min(0,e)f(Y)=\mathlarger{\sum}_{\overline{\lambda} \in \CR_1^{0+}}C_{\overline{\lambda}}\dfrac{\alpha(1_1,1_1)}{\alpha(A_{\overline{\lambda}},A_{\overline{\lambda}})}(-q)^{\CB_{\overline{\lambda}}(Y)}f(Y).
\end{equation}

Note that in \eqref{eq4.1.1} and \eqref{eq4.1.2}, these two functions are zero for $Y \in \FD_0$. Therefore, one can regard these as functions on $Y$ with only one $e$. Also, running $\lambda$ over $\FD_1 \cup \FD_2$ is equivalent to running $\overline{\lambda}$ over $\CR_1^{0+}$ (by removing 1 zero). Therefore, by \eqref{eq4.1.3}, we have

\begin{equation*}
	\begin{array}{l}
		\CF_0'(Y,A_1^{[0]})-\dfrac{\CF_0'(Y,A_0^{[0]})}{(-q)^{2}}\\\\
		=\mathlarger{\sum}_{\lambda \in \FD_1 \cup \FD_2} C_{\overline{\lambda}}\dfrac{\alpha(1_1,1_1)}{\alpha(A_{\overline{\lambda}},A_{\overline{\lambda}})}(-q)^{-4}\lbrace \CF_0(Y,A_{\lambda_1^+})-(-q)^2\CF_0(Y,A_{\lambda})\rbrace.
		\end{array}
\end{equation*}
\end{proof}

This proposition implies the following proposition.

\begin{proposition}\label{proposition4.2} For $B \in X_{2n}(O_E)$, we have
	\begin{equation*}
		\dfrac{W'_{0,1}(B,0)}{W'_{1,1}(A_1,0)}=\mathlarger{\sum}_{\lambda \in \CR_{2}^{0+}}D_{\lambda}\dfrac{\alpha(A_{\lambda},B)}{\alpha(A_{\lambda},A_{\lambda})}.
	\end{equation*}
Here, we have
\begin{equation*}
	D_{\lambda}=\left\lbrace \begin{array}{ll}
		-(-1)^{\sum \lambda_i}(q^2-1)& \text{ if }E_0(\lambda)=E_1(\lambda)=0\\
		(-1)^{\sum \lambda_i} & \text{ if }E_0(\lambda)=1, E_1(\lambda)=0\\
		-(-1)^{\sum \lambda_i} & \text{ if }E_0(\lambda)=0, E_1(\lambda)=1\\
		-(q-1) & \text{ if }E_0(\lambda)=0, E_1(\lambda)=2\\
		-(q+2)/(q+1) & \text{ if }E_0(\lambda)=1, E_1(\lambda)=1\\
		 1/(q+1)& \text{ if }E_0(\lambda)=2, E_1(\lambda)=0\\
	
	\end{array}\right..
\end{equation*}

\end{proposition}
\begin{proof}
	By Proposition \ref{proposition4.1} and
	\begin{equation*}
		\begin{array}{l}
			W_{1,1}(A_1,0)=(q+1)^2/q^5,\\
			\alpha(1_2,1_2)=(q+1)(q^2-1)/q^3,\\
			\alpha(1_1,1_1)=(q+1)/q,\\
			\alpha(\pi^{k},\pi^{k})=q^{k-1}(q+1),
			\end{array}
	\end{equation*}
we have
\begin{equation*}\begin{array}{l}
	\dfrac{\CF_0'(Y,A_1)}{W_{1,1}(A_1,0)}\\
	=(q-1)\dfrac{\CF_0'(Y,A_0)}{\alpha(1_2,1_2)}\\
 +\dfrac{q^5}{(q+1)^2}\mathlarger{\sum}_{\lambda \in \FD_1 \cup \FD_2} C_{\overline{\lambda}}\dfrac{\alpha(1_1,1_1)}{\alpha(A_{\overline{\lambda}},A_{\overline{\lambda}})}(-q)^{-4}\lbrace \CF_0(Y,A_{\lambda_1^+})-(-q)^2\CF_0(Y,A_{\lambda})\rbrace\\
 =(q-1)\mathlarger{\sum}_{\lambda \in \CR_2^{0+}} C_{\lambda}\dfrac{\CF_0(Y,A_{\lambda})}{\alpha(A_{\lambda},A_{\lambda})}+\mathlarger{\sum}_{\lambda \in \FD_1 \cup \FD_2} \dfrac{C_{\overline{\lambda}}\CF_0(Y,A_{\lambda_1^+})}{q^{\overline{\lambda}-1}(q+1)^2}-\dfrac{C_{\overline{\lambda}}\CF_0(Y,A_{\lambda})}{q^{\overline{\lambda}-3}(q+1)^2}.
	\end{array}
\end{equation*}

Combining this with Proposition \ref{proposition3.3} and Proposition \ref{proposition3.4}, we can compute $D_{\lambda}$ as follows.

$\bullet$ If $E_0(\lambda)=E_1(\lambda)=0$, then
\begin{equation*}
D_{\lambda}=(q-1)C_{\lambda}=-(-1)^{\sum \lambda_i}(q^2-1).
\end{equation*}

$\bullet$ If $E_0(\lambda)=0, E_1(\lambda)=1$, i.e., $\lambda=(\overline{\lambda},1)$ for $\overline{\lambda}\geq 2$, we have
\begin{equation*}
	D_{\lambda}=(q-1)C_{\lambda}+\dfrac{C_{\overline{\lambda}}\alpha(A_{\lambda},A_{\lambda})}{q^{\overline{\lambda}-1}(q+1)^2}.
\end{equation*}
Since $\alpha(A_{\lambda},A_{\lambda})=q^{\overline{\lambda}+1}(q+1)^2$, $C_{\overline{\lambda}}=(-1)^{\overline{\lambda}+1}$, and $C_{\lambda}=(-1)^{\overline{\lambda}+2}(q+1)$, we have
\begin{equation*}
	D_{\lambda}=(-1)^{\overline{\lambda}}((q^2-1)-q^2)=(-1)^{\overline{\lambda}+1}.
\end{equation*}

$\bullet$ If $E_0(\lambda)=1, E_1(\lambda)=0$, i.e., $\lambda=(\overline{\lambda},0)$ for $\overline{\lambda}\geq2$, we have
\begin{equation*}
	D_{\lambda}=(q-1)C_{\lambda}-\dfrac{C_{\overline{\lambda}}\alpha(A_{\lambda},A_{\lambda})}{q^{\overline{\lambda}-3}(q+1)^2}.
\end{equation*}
Since $\alpha(A_{\lambda},A_{\lambda})=q^{\overline{\lambda}-2}(q+1)^2$, $C_{\overline{\lambda}}=(-1)^{\overline{\lambda}+1}$, and $C_{\lambda}=(-1)^{\overline{\lambda}+1}$, we have
\begin{equation*}
	D_{\lambda}=(-1)^{\overline{\lambda}+1}((q-1)-q)=(-1)^{\overline{\lambda}}.
\end{equation*}

$\bullet$ If $E_0(\lambda)=0, E_1(\lambda)=2$, i.e., $\lambda=(1,1)$, we have
\begin{equation*}
	D_{\lambda}=(q-1)C_{\lambda}+\dfrac{C_{\overline{\lambda}}\alpha(A_{\lambda},A_{\lambda})}{q^{\overline{\lambda}-1}(q+1)^2}.
\end{equation*}
Since $\alpha(A_{\lambda},A_{\lambda})=q(q+1)(q^2-1)$, $C_{\overline{\lambda}}=1$, and $C_{\lambda}=-(q+1)$, we have
\begin{equation*}
	D_{\lambda}=-(q^2-1)+q(q-1)=-(q-1).
\end{equation*}

$\bullet$ If $E_0(\lambda)=1, E_1(\lambda)=1$, i.e., $\lambda=(1,0)$, we have
\begin{equation*}
	D_{\lambda}=(q-1)C_{\lambda}+\dfrac{C_{0}\alpha(A_{\lambda},A_{\lambda})}{q^{0-1}(q+1)^2}-\dfrac{C_{1}\alpha(A_{\lambda},A_{\lambda})}{q^{1-3}(q+1)^2}.
\end{equation*}
Here, $\alpha(A_{\lambda},A_{\lambda})=\dfrac{(q+1)^2}{q}$, $C_{\overline{\lambda}}=1$, $C_1=1$, and $C_0=\dfrac{\alpha'(1_1,1_1)}{\alpha(1_1,1_1)}=-\dfrac{1}{q+1}$.

Therefore,
\begin{equation*}
	D_{\lambda}=(q-1)-\dfrac{1}{q+1}-q=-\dfrac{q+2}{q+1}.
\end{equation*}

$\bullet$ If $E_0(\lambda)=2, E_1(\lambda)=0$, i.e., $\lambda=(0,0)$, we have
\begin{equation*}
		D_{\lambda}=(q-1)C_{\lambda}-\dfrac{C_{0}\alpha(1_2,1_2)}{q^{0-3}(q+1)^2}.
\end{equation*}
Here $\alpha(1_2,1_2)=\dfrac{(q+1)(q^2-1)}{q^3}$, $C_{\lambda}=\dfrac{\alpha'(1_2,1_2)}{\alpha(1_2,1_2)}=-\dfrac{q-2}{q^2-1}$, and $C_0=\dfrac{\alpha'(1_1,1_1)}{\alpha(1_1,1_1)}=-\dfrac{1}{q+1}$.

Therefore,
\begin{equation*}
	D_{\lambda}=-\dfrac{q-2}{q+1}+\dfrac{q-1}{q+1}=\dfrac{1}{q+1}.
\end{equation*}

This finishes the proof of the proposition.
\end{proof}

Now, we want to generalize this toy example to arbitrary $n$. We will do a similar computation, but \eqref{eq.4.1.0} shows that we need more complicated objects than $n=1$ case.

Recall that we use functions $\CF_0(Y,A_{\lambda_1^+})-(-q)^2\CF_0(Y,A_{\lambda})$, $\lambda \in \FD_1 \cup \FD_2$ to express $\CF_0'(Y,A_1^{[0]})-(-q)^{-2}\CF_0'(Y,A_0^{[0]})$. Note that $\CF_0(Y,A_{\lambda_1^+})-(-q)^2\CF_0(Y,A_{\lambda})$ is $0$ for $Y \in D_0$. In general, for $0\leq l \leq 2n$ and $\lambda \in \bigcup_{k=l}^{2n} \FD_k$, we need a function such that it is $0$ for $Y \in \bigcup_{k=0}^{l-1} \FD_k$. For this, we define the following notation:

For $\lambda=(\overline{\lambda},\overset{E_1(\lambda)}{\overbrace{1,\dots,1}},\overset{E_0(\lambda)}{\overbrace{0,\dots,0}}) \in \FD_l$, and $0 \leq s \leq l$, we define
\begin{equation*}
	\lambda_s^+:=(\overline{\lambda},\overset{E_1(\lambda)+s}{\overbrace{1,\dots,1}},\overset{E_0(\lambda)-s}{\overbrace{0,\dots,0}}),,
\end{equation*}
by replacing $s$ zeros by $s$ 1's.

Note that $\CF_0(Y,A_{\lambda_s^+})=(-q)^{s\CB_1(Y)}\CF_0(Y,A_{\lambda})$. From this we can consider the following matrix with entries $(-q)^{s\CB_1(Y)}$: for $\lambda  \in \bigcup_{k=l}^{2n} \FD_k$
\begin{equation*}
	\begin{array}{c|ccccccc}
		Y \backslash s & \lambda & \lambda_1^+ & \lambda_2^+ &\dots &  \lambda_s^+  & \dots &  \lambda_l^+\\
		\hline\\
		\FD_0 &1&(-q)^{2n}	&(-q)^{4n}&\dots & (-q)^{2sn}&\dots &(-q)^{2ln}\\
		\FD_1 &1&(-q)^{2n-1}	& (-q)^{2(2n-1)}&\dots&(-q)^{s(2n-1)}&\dots&(-q)^{l(2n-1)}\\
		\FD_2 &1&(-q)^{2n-2}	& (-q)^{2(2n-2)}&\dots&(-q)^{s(2n-2)}&\dots&(-q)^{l(2n-2)}\\\\
		\vdots &1&\vdots	&\vdots&\vdots&\vdots&\vdots&\vdots\\\\
		\FD_k &1&(-q)^{2n-k}	&(-q)^{2(2n-k)}&\dots&(-q)^{s(2n-k)}&\dots&(-q)^{l(2n-k)}\\\\
		\vdots &1&\vdots	&\vdots&\vdots&\vdots&\vdots&\vdots\\\\
		\FD_{2n} &1&1	&1	&\dots&1&\dots&1
	\end{array}
\end{equation*}

Let us define $(l+1) \times (l+1)$ matrices $\FM_l$, $\FX_l$, $\delta_l$ as follows:
\begin{equation*}
	\FM_{l}:\left(\begin{array}{cccc}
		1 & (-q)^{2n} & \dots & (-q)^{2ln}\\
		1 & (-q)^{2n-1} & \dots & (-q)^{l(2n-1)}\\
		\vdots & \vdots & \ddots & \vdots\\
		1 & (-q)^{2n-l} & \dots & (-q)^{l(2n-l)}
	\end{array}\right),
\end{equation*}
\begin{equation*}
	\FX_{l}:\left(\begin{array}{cccc}
		1 & 1 & \dots & 1\\
		1 & (-q)^{-1} & \dots & (-q)^{-l}\\
		\vdots & \vdots & \ddots & \vdots\\
		1 & (-q)^{-l} & \dots & (-q)^{-l^2}
	\end{array}\right),
\end{equation*}
and 
  \begin{equation*}
  	\delta_{l}:\left(\begin{array}{cccc}
  		1 & 0 & \dots &0\\
  		0 & (-q)^{2n} & \dots & 0\\
  		\vdots & \vdots & \ddots & \vdots\\
  		0 & 0 & \dots & (-q)^{2ln}
  	\end{array}\right)=\diag(1,(-q)^{2n},\dots,(-q)^{2ln}).
  \end{equation*}

Note that
\begin{equation*}
	\FM_l=\FX_l\delta_l.
\end{equation*}

Now, let us define constants $d_{il}$, $0 \leq i \leq l$, as follows.
\begin{equation*}
\left(	\begin{array}{l}
		d_{0l}\\
		d_{1l}\\
		d_{2l}\\
		\vdots\\
		d_{ll}
		\end{array}\middle)=\FM_l^{-1}\middle(\begin{array}{l}
		0\\
		0\\
		0\\
		\vdots\\
		1
	\end{array}\right).
\end{equation*}

Also, for $0 \leq j,l \leq 2n$, let us consider the following constants $\CA_{jl}$:
\begin{equation*}
\left(\begin{array}{c}
	\CA_{0l}\\
	\CA_{1l}\\
	\CA_{2l}\\
	\vdots\\
	\CA_{ll}\\
	\CA_{l+1,l}\\
	\vdots\\
	\CA_{2n,l}
\end{array}\middle)=\FM_{2n}\middle( \begin{array}{l}d_{0l}\\
d_{1l}\\
d_{2l}\\
\vdots\\
d_{ll}\\
0\\
\vdots\\
0
\end{array}\right).
\end{equation*}

Then, for $\lambda \in \bigcup_{k=l}^{2n} \FD_k$, we have that the function $\mathlarger{\sum}_{0 \leq i\leq l}d_{il}\CF_0(Y,A_{\lambda_{i}^+})$ satisfies
\begin{equation}\label{eq.4.1.5}
	\mathlarger{\sum}_{0 \leq i\leq l} d_{il}\CF_0(Y,A_{\lambda_{i}^+})=\left\lbrace \begin{array}{ll}
		0 =\CA_{0l}\CF_0(Y,A_{\lambda}) & \text{ if } Y \in \FD_0\\
		\quad\quad	\vdots&\quad	\vdots\\
		0=\CA_{l-1,l}\CF_0(Y,A_{\lambda}) & \text{ if } Y \in \FD_{l-1}\\
		\CF_0(Y,A_{\lambda}) =\CA_{ll}\CF_0(Y,A_{\lambda}) & \text{ if } Y \in \FD_{l}\\
		\CA_{kl} \CF_0(Y,A_{\lambda}) & \text{ if } Y \in \FD_{k}, l+1 \leq k \leq 2n
	\end{array}\right.
\end{equation}

Recall from Proposition \ref{proposition3.3} and Proposition \ref{proposition3.4} that for $Y$
\begin{equation*}
	Y=\sigma\left( \begin{array}{lll}
		\pi^{e_1} & & 0\\
		 & \ddots&  \\
		 0& & \pi^{e_{2n-l}}
	\end{array}\right),
\end{equation*}
we have
\begin{equation*}
	\begin{array}{c}
		\dfrac{\CF_0'(Y,1_{2n-l})}{\alpha(1_{2n-l},1_{2n-l})}=\mathlarger{\sum}_{\overline{\lambda} \in \CR_{2n-l}^{0+}}C_{\overline{\lambda}}\dfrac{\CF_0(Y,A_{\overline{\lambda}})}{\alpha(A_{\overline{\lambda}},A_{\overline{\lambda}})}\\
		\Longleftrightarrow \mathlarger{\sum}_{j} \min(0,e_j)f(Y)=\mathlarger{\sum}_{\overline{\lambda} \in \CR_{2n-l}^{0+}}C_{\overline{\lambda}}\dfrac{\alpha(1_{2n-l},1_{2n-l})}{\alpha(A_{\overline{\lambda}},A_{\overline{\lambda}})}(-q)^{\CB_{\overline{\lambda}}(Y)}f(Y).
		\end{array}
\end{equation*}

For $\lambda \in \bigcup_{k=l}^{2n}\FD_k$, let us define $\lambda^{\vee_l}$ as the element in $\CR_{2n-l}^{0+}$ such that $\lambda=(\lambda^{\vee_l},0,\dots,0)$.

Since running $\lambda=(\overline{\lambda},0,\dots,0)$ over $\bigcup_{k=l}^{2n}\FD_k$ is equivalent to running $\lambda^{\vee_l}$ over $\CR_{2n-l}^{0+}$ by removing $l$ zeros, we have
\begin{equation}
	\mathlarger{\sum}_{\lambda \in \bigcup_{k=l}^{2n}\FD_k} C_{\lambda^{\vee_l}}\dfrac{\alpha(1_{2n-l},1_{2n-l})}{\alpha(A_{\lambda^{\vee_l}},A_{\lambda^{\vee_l}})}\mathlarger{\sum}_{0 \leq i\leq l} d_{il}\CF_0(Y,A_{\lambda_{i}^+})=\left\lbrace \begin{array}{l}
	0  \\ 
	\text{ if } Y \in \FD_i,\\
	\text{ } 0 \leq i \leq l-1;\\\\
	\CA_{il}\sum_{j}\min(0,e_j)f(Y),\\
	\text{ if } Y \in \FD_i,\\
	\text{ } l \leq i \leq 2n.
	\end{array}\right.
\end{equation}

Now, let us go back to \eqref{eq.4.1.0}. In \eqref{eq.4.1.0}, we have
\begin{equation*}
	\begin{array}{l}
		\CF_0'(Y,A_n^{[0]})-\dfrac{\CF_0'(Y,A_0^{[0]})}{(-q)^{2n^2}}\\
		=\left\lbrace \begin{array}{cl} 
			0 & \text{ if } Y \in \FD_0\\
			(\mathlarger{\sum}_{j}\min(0,e_j))f(Y)((-q)^{n(2n-1)-4n^2}-(-q)^{-2n^2})& \text{ if } Y \in \FD_1\\
			\quad\quad\quad\quad\vdots\\
			(\mathlarger{\sum}_{j}\min(0,e_j))f(Y)((-q)^{n(2n-k)-4n^2}-(-q)^{-2n^2})& \text{ if } Y \in \FD_k\\
			\quad\quad\quad\quad\vdots\\
			(\mathlarger{\sum}_{j}\min(0,e_j))f(Y)((-q)^{-4n^2}-(-q)^{-2n^2}) &\text{ if }Y \in \FD_{2n}	
		\end{array}\right.
	\end{array}
\end{equation*}

Therefore, let us define the constants $\CK_l$, $0\leq l \leq 2n$, such that
\begin{equation}\label{eq.4.1.7}
\left( \begin{array}{llll}
	\CA_{00} & \CA_{01} & \dots &\CA_{02n}\\
	\CA_{10} & \CA_{11} & \dots & \CA_{1 2n}\\
	\vdots & \ & \ddots&\vdots\\
	\CA_{2n0} & \CA_{2n1} & \dots & \CA_{2n2n}
	\end{array}\middle)\middle(\begin{array}{l}
	\CK_0\\
	\CK_1\\
	\vdots\\
	\CK_k\\
	\vdots\\
	\CK_{2n}
	\end{array}\middle)=\middle(\begin{array}{c}
	0\\
	(-q)^{n(2n-1)-4n^2}-(-q)^{-2n^2}\\
	\vdots\\
	(-q)^{n(2n-k)-4n^2}-(-q)^{-2n^2}\\
	\vdots\\
		(-q)^{-4n^2}-(-q)^{-2n^2}
\end{array}\right),
\end{equation}
Then, we have
\begin{equation}\label{eq4.1.8}\begin{array}{l}
	\CF_0'(Y,A_n^{[0]})-\dfrac{\CF_0'(Y,A_0^{[0]})}{(-q)^{2n^2}}\\\\
	=\mathlarger{\sum}_{0 \leq l \leq 2n} \CK_l \mathlarger{\mathlarger{\mathlarger{\mathlarger{\lbrace}}}}\mathlarger{\sum}_{\lambda \in \bigcup_{k=l}^{2n}\FD_k} C_{\lambda^{\vee_l}}\dfrac{\alpha(1_{2n-l},1_{2n-l})}{\alpha(A_{\lambda^{\vee_l}},A_{\lambda^{\vee_l}})}\lbrace\mathlarger{\sum}_{0 \leq i\leq l} d_{il}\CF_0(Y,A_{\lambda_{i}^+})\rbrace\mathlarger{\mathlarger{\mathlarger{\mathlarger{\rbrace}}}}.
	\end{array}
\end{equation}

Now, we need to find all constants $\CK_l$, $d_{il}$ precisely.

\begin{lemma}\label{lemma4.3}
	Let $\Delta$ be the upper triangular $2n+1 \times 2n+1$ matrix
	\begin{equation*}
		\Delta=\left( \begin{array}{ccccc}
			d_{00} & d_{01} & d_{02} & \dots &d_{02n}\\
			0 & d_{11} & d_{12} & \dots & d_{12n}\\
			0 &  0& d_{22} &\dots & d_{22n}\\
			\vdots & \ddots & \ddots & \ddots& \vdots \\
			0 & 0 & 0 & \dots & d_{2n2n} \\
			\end{array}\right).
			\end{equation*}
		Then
		\begin{equation*}
			\Delta\left(\begin{array}{c}
				\CK_0\\
				\CK_1\\
				\dots \\
				\CK_{n}\\
				\dots\\
				\CK_{2n}
			\end{array}\middle)=\middle(\begin{array}{c}
			-(-q)^{-2n^2}\\
			0\\
			0\\
			(-q)^{-4n^2}\\
			0\\
			0
			\end{array} \right)\begin{array}{l} 
			\text{1st entry}\\
			\\
			\\
			\text{(n+1)-th entry}
			\\
			\\
			\\
			\end{array}
		\end{equation*}
\end{lemma}
\begin{proof}
	In \eqref{eq.4.1.7}, we have
	\begin{equation*}
		\left( \begin{array}{llll}
			\CA_{00} & \CA_{01} & \dots &\CA_{02n}\\
			\CA_{10} & \CA_{11} & \dots & \CA_{1 2n}\\
			\vdots & \ & \ddots&\vdots\\
			\CA_{2n0} & \CA_{2n1} & \dots & \CA_{2n2n}
		\end{array}\middle)\middle(\begin{array}{l}
			\CK_0\\
			\CK_1\\
			\vdots\\
			\CK_k\\
			\vdots\\
			\CK_{2n}
		\end{array}\middle)=\middle(\begin{array}{c}
			0\\
			(-q)^{n(2n-1)-4n^2}-(-q)^{-2n^2}\\
			\vdots\\
			(-q)^{n(2n-k)-4n^2}-(-q)^{-2n^2}\\
			\vdots\\
			(-q)^{-4n^2}-(-q)^{-2n^2}
		\end{array}\right).
	\end{equation*}

Also, recall that
\begin{equation*}
	\FM_{2n} \Delta=\left( \begin{array}{llll}
		\CA_{00} & \CA_{01} & \dots &\CA_{02n}\\
		\CA_{10} & \CA_{11} & \dots & \CA_{1 2n}\\
		\vdots & \ & \ddots&\vdots\\
		\CA_{2n0} & \CA_{2n1} & \dots & \CA_{2n2n}
	\end{array}\right).
\end{equation*}

Therefore, we have
\begin{equation*}
	\Delta\left(\begin{array}{c}
		\CK_0\\
		\CK_1\\
		\dots \\
		\CK_{n}\\
		\dots\\
		\CK_{2n}
	\end{array}\middle)=\FM_{2n}^{-1}\middle(\begin{array}{c}
	0\\
	(-q)^{n(2n-1)-4n^2}-(-q)^{-2n^2}\\
	\vdots\\
	(-q)^{n(2n-k)-4n^2}-(-q)^{-2n^2}\\
	\vdots\\
	(-q)^{-4n^2}-(-q)^{-2n^2}
\end{array}\right).
\end{equation*}

Now, it is easy to check that
\begin{equation*}
	\FM_{2n}\left(\begin{array}{c}
		-(-q)^{-2n^2}\\
		0\\
		0\\
		(-q)^{-4n^2}\\
		0\\
		0
	\end{array} \middle)=\middle(\begin{array}{c}
		0\\
		(-q)^{n(2n-1)-4n^2}-(-q)^{-2n^2}\\
		\vdots\\
		(-q)^{n(2n-k)-4n^2}-(-q)^{-2n^2}\\
		\vdots\\
		(-q)^{-4n^2}-(-q)^{-2n^2}
	\end{array}\right).
\end{equation*}
This finishes the proof of the lemma.
\end{proof}

\begin{lemma}\label{lemma4.4}
	For $0 \leq i \leq l$, we have
	\begin{equation*}
		d_{il}=(-q)^{-2in} \prod_{\substack{0 \leq m \leq l\\m \neq i}}\dfrac{1}{((-q)^{-i}-(-q)^{-m})},
	\end{equation*}
and hence
\begin{equation*}
	\dfrac{d_{il}}{d_{i+1,l}}=-(-q)^{n+1}\dfrac{(-q)^n-(-q)^{n-i-1}}{(-q)^{l-i}-1}.
\end{equation*}

\end{lemma}
\begin{proof}
	Recall that
		\begin{equation*}
			\left(	\begin{array}{l}
				d_{0l}\\
				d_{1l}\\
				d_{2l}\\
				\vdots\\
				d_{ll}
			\end{array}\middle)=\FM_l^{-1}\middle(\begin{array}{l}
				0\\
				0\\
				0\\
				\vdots\\
				1
			\end{array}\right),
		\end{equation*}
	and
	\begin{equation*}
		\FM_l=\FX_l\delta_l.
	\end{equation*}

Therefore, $d_{il}$ is the $(i+1)$-th entry of the last column of $\FM_l^{-1}=\delta_l^{-1}\FX_l^{-1}$. Since $\FX_l$ is a Vandermonde matrix and $\delta_l$ is a diagonal matrix, we can compute that
\begin{equation*}
	d_{il}=(-q)^{-2in} \prod_{\substack{0 \leq m \leq l\\m \neq i}}\dfrac{1}{((-q)^{-i}-(-q)^{-m})}.
\end{equation*}
	The statement on $d_{il}/d_{i+1,l}$ follows from this.
\end{proof}

\begin{lemma}\label{lemma4.5}
	For $i=0$ and $n+1 \leq i \leq 2n$, $\CK_i=0$. Also, $d_{nn}\CK_n=(-q)^{-4n^2}$ and for $1 \leq l \leq n-1$, we have
	\begin{equation*}\begin{array}{ll}
		d_{n-l,n-l}\CK_{n-l}&=(-q)^{n+l}\dfrac{(-q)^{n}-(-q)^{l-1}}{(-q)^{l}-1}d_{n-l+1,n-l+1}\CK_{n-l+1}
		\end{array}
	\end{equation*}
\end{lemma}

\begin{proof}
	By Lemma \ref{lemma4.3}, we have
	\begin{equation*}
		\Delta\left(\begin{array}{c}
			\CK_0\\
			\CK_1\\
			\dots \\
			\CK_{n}\\
			\dots\\
			\CK_{2n}
		\end{array}\middle)=\middle(\begin{array}{c}
			-(-q)^{-2n^2}\\
			0\\
			0\\
			(-q)^{-4n^2}\\
			0\\
			0
		\end{array} \right)\begin{array}{l} 
			\text{1st entry}\\
			\\
			\\
			\text{(n+1)-th entry}
			\\
			\\
			\\
		\end{array}
	\end{equation*}
Since $\Delta$ is an upper triangular matrix, we have that
$\CK_i=0$ for $n+1 \leq i \leq 2n$. Also, the $(n+1)$-th row of the matrix implies that
\begin{equation*}
	d_{nn}\CK_n+d_{n,n+1}\CK_{n+1}+\dots+d_{n,2n}\CK_{2n}=(-q)^{-4n^2},
\end{equation*}
and hence $d_{nn}\CK_n=(-q)^{-4n^2}$.

The $n$-th row of the matrix implies that
\begin{equation*}
	d_{n-1,n-1}\CK_{n-1}+d_{n-1n}\CK_n=0.
\end{equation*}

From Lemma \ref{lemma4.4}, we know that
\begin{equation*}
	d_{n-1,n}\CK_n=\dfrac{d_{n-1,n}}{d_{nn}}d_{nn}\CK_n=-(-q)^{n+1}\dfrac{(-q)^n-1}{(-q)-1}d_{nn}\CK_n.
\end{equation*}
Therefore, we have
\begin{equation*}
	d_{n-1,n-1}\CK_{n-1}=(-q)^{n+1}\dfrac{(-q)^n-1}{(-q)-1}d_{nn}\CK_n.
\end{equation*}

Now assume that
\begin{equation}\label{eq4.1.9}
	d_{n-j,n-j}\CK_{n-j}=(-q)^{n+j}\dfrac{(-q)^{n}-(-q)^{j-1}}{(-q)^{j}-1}d_{n-j+1,n-j+1}\CK_{n-j+1}.
\end{equation}
for $1 \leq j \leq k-1$.

Then, by the $(n-k+1)$-th row of the matrix, we have
\begin{equation*}
	d_{n-k,n-k}\CK_{n-k}+d_{n-k,n-k+1}\CK_{n-k+1}+\dots+d_{n-k,n}\CK_{n}=0.
\end{equation*}

By using \eqref{eq4.1.9}, we have
\begin{equation*}
	\begin{array}{l}
	d_{n-k,n-k}\CK_{n-k}\\
	=-d_{n-k,n-k+1}\CK_{n-k+1}-\dots-d_{n-k,n}\CK_{n}\\
	=-\dfrac{d_{n-k,n-k+1}}{d_{n-k+1,n-k+1}}d_{n-k+1,n-k+1}\CK_{n-k+1}\\
	-\dfrac{d_{n-k,n-k+2}}{d_{n-k+1,n-k+2}}\dfrac{d_{n-k+1,n-k+2}}{d_{n-k+2,n-k+2}}d_{n-k+2,n-k+2}\CK_{n-k+2}\\
	\dots\\
	=d_{n-k+1,n-k+1}\CK_{n-k+1} \times\\
	\mathlarger{\mathlarger{\mathlarger{\mathlarger{\lbrace}}}} (-q)^{n+1}\dfrac{(-q)^n-(-q)^{k-1}}{(-q)-1}\\
	-(-q)^{2(n+1)}\dfrac{(-q)^n-(-q)^{k-1}}{(-q)-1}\dfrac{(-q)^n-(-q)^{k-2}}{(-q)^2-1}\dfrac{(-q)^{k-1}-1}{(-q)^n-(-q)^{k-2}}\dfrac{1}{(-q)^{n+k-1}}
	\\
	\dots \mathlarger{\mathlarger{\mathlarger{\mathlarger{\rbrace}}}}
\end{array}
\end{equation*}
\begin{equation*}
\begin{array}{ll}
	=d_{n-k+1,n-k+1}\CK_{n-k+1} \times (-q)^{n+k}\dfrac{(-q)^n-(-q)^{k-1}}{(-q)^k-1}\times &\quad\quad\quad\quad\quad\quad\quad\quad\\
	\mathlarger{\mathlarger{\mathlarger{\mathlarger{\lbrace}}}} 
	\dfrac{(-q)^k-1}{(-q)^{k-1}((-q)-1)}-\dfrac{((-q)^k-1)((-q)^{k-1}-1)}{(-q)^{2k-3}((-q)-1)((-q)^2-1)}\\
	+\dfrac{((-q)^k-1)((-q)^{k-1}-1)((-q)^{k-2}-1)}{(-q)^{3k-6}((-q)-1)((-q)^2-1)((-q)^3-1)}-
	\\
	\dots \mathlarger{\mathlarger{\mathlarger{\mathlarger{\rbrace}}}}&\\
	=d_{n-k+1,n-k+1}\CK_{n-k+1} \times (-q)^{n+k}\dfrac{(-q)^n-(-q)^{k-1}}{(-q)^k-1}\times&\\
	\mathlarger{\sum}_{1 \leq j\leq k}(-1)^{j-1}\mathlarger{\prod}_{1 \leq m \leq j} \dfrac{((-q)^{-m+1}-(-q)^{-k})}{(1-(-q)^{-m})}&.
	\end{array}
\end{equation*}

Now, it suffices to show that
\begin{equation*}
\mathlarger{\sum}_{1 \leq j\leq k}(-1)^{j-1}\mathlarger{\prod}_{1 \leq m \leq j} \dfrac{((-q)^{-m+1}-(-q)^{-k})}{(1-(-q)^{-m})}=1
\end{equation*}

Note that the sum of $k-1$-th and $k$-th term is
\begin{equation*}\begin{array}{l}
	(-1)^{k-2}\mathlarger{\lbrace}\mathlarger{\prod}_{1 \leq m \leq k-1} \dfrac{((-q)^{-m+1}-(-q)^{-k})}{(1-(-q)^{-m})}-\mathlarger{\prod}_{1 \leq m \leq k} \dfrac{((-q)^{-m+1}-(-q)^{-k})}{(1-(-q)^{-m})}\mathlarger{\rbrace}\\
	=(-1)^{k-2}\mathlarger{\lbrace}\mathlarger{\prod}_{1 \leq m \leq k-1} \dfrac{((-q)^{-m+1}-(-q)^{-k})}{(1-(-q)^{-m})}(1-\dfrac{((-q)^{-k+1}-(-q)^{-k})}{(1-(-q)^{-k})})\mathlarger{\rbrace}\\
	=(-1)^{k-2}\mathlarger{\lbrace}\mathlarger{\prod}_{1 \leq m \leq k-1} \dfrac{((-q)^{-m+1}-(-q)^{-k})}{(1-(-q)^{-m})}(\dfrac{(1-(-q)^{-k+1})}{(1-(-q)^{-k})})\mathlarger{\rbrace}\\
	=(-1)^{k-2}\mathlarger{\lbrace}\mathlarger{\prod}_{1 \leq m \leq k-2} \dfrac{((-q)^{-m+1}-(-q)^{-k})}{(1-(-q)^{-m})}(\dfrac{((-q)^{-k+2}-(-q)^{-k})}{(1-(-q)^{-k})})\mathlarger{\rbrace}.
	\end{array}
\end{equation*}

Then, if we add $k-2$-th term, we have
\begin{equation*}\begin{array}{l}
	(-1)^{k-3}\mathlarger{\lbrace}\mathlarger{\prod}_{1 \leq m \leq k-2} \dfrac{((-q)^{-m+1}-(-q)^{-k})}{(1-(-q)^{-m})}(1-\dfrac{((-q)^{-k+2}-(-q)^{-k})}{(1-(-q)^{-k})})\mathlarger{\rbrace}\\
	=(-1)^{k-3}\mathlarger{\lbrace}\mathlarger{\prod}_{1 \leq m \leq k-2} \dfrac{((-q)^{-m+1}-(-q)^{-k})}{(1-(-q)^{-m})}(\dfrac{(1-(-q)^{-k+2})}{(1-(-q)^{-k})})\mathlarger{\rbrace}\\
	=(-1)^{k-3}\mathlarger{\lbrace}\mathlarger{\prod}_{1 \leq m \leq k-3} \dfrac{((-q)^{-m+1}-(-q)^{-k})}{(1-(-q)^{-m})}(\dfrac{((-q)^{-k+3}-(-q)^{-k})}{(1-(-q)^{-k})})\mathlarger{\rbrace}.
	\end{array}
\end{equation*}

If we continue this, we have
\begin{equation*}
	(-1)^{1-1}\dfrac{(1-(-q)^{-k})}{(1-(-q)^{-1})}\dfrac{(1-(-q)^{-1})}{(1-(-q)^{-k})}=1
\end{equation*}

For $\CK_0$, we can use the first row of the matrix
\begin{equation*}
	d_{00}\CK_0+d_{01}\CK_1+\dots+d_{0n}\CK_n=-(-q)^{-2n^2},
\end{equation*}
and as above, we can show that $\CK_0=0$. This finishes the proof of the lemma.
\end{proof}

\begin{lemma}\label{lemma4.6}
	Let $B$ be a diagonal matrix
	\begin{equation*}
		B=\left( \begin{array}{llll}
			\pi^{a_1}1_{k_1} &&&\\
			& \pi^{a_2}1_{k_2}& &\\
				&& \ddots &\\
				&&&\pi^{a_t} 1_{k_t}
		\end{array}\right),
	\end{equation*}
where $a_1 > a_2 > \dots >a_t$. For $1 \leq i  \leq t$, let $n_i=\sum_{1 \leq j \leq i} k_j$. Then, we have
\begin{equation*}
	\alpha(B,B)=\prod_{1 \leq j \leq t}q^{a_j(n_j^2-n^2_{j-1})}\prod_{1 \leq i \leq t}\lbrace \prod_{1 \leq m \leq k_i}(1-(-q)^{-m})\rbrace.
\end{equation*}

\end{lemma}
\begin{proof} First, we have that
	\begin{equation*}
		\alpha(B,B)=q^{a_tn_t^2}\alpha(\pi^{-a_t}B,\pi^{-a_t}B).
	\end{equation*}
Then by \cite[Corollary 9.12]{KR2}, we have that
\begin{equation*}
	q^{a_tn_t^2}\alpha(\pi^{-a_t}B,\pi^{-a_t}B)=q^{a_tn_t^2}\alpha(\pi^{-a_t}B,1_{a_t})\alpha(B_1,B_1),
\end{equation*}
where $B_1$ is the matrix
\begin{equation*}
		B_1=\left( \begin{array}{llll}
		\pi^{a_1-a_t}1_{k_1} &&&\\
		& \pi^{a_2-a_t}1_{k_2}& &\\
		&& \ddots &\\
		&&&\pi^{a_{t-1}-a_t} 1_{k_{t-1}}
	\end{array}\right).
\end{equation*}

One can prove that for a $k \times k$ integral hermitian matrix $C$,
\begin{equation*}
	\alpha(\left(\begin{array}{ll}
		\pi C & 0\\
		0 & 1_{m-k}
	\end{array}\right),1_{n})=\alpha(\left(\begin{array}{ll}
	\pi 1_{k} & 0\\
	0 & 1_{m-k}
\end{array}\right),1_{n}).
\end{equation*}
This can be proved by using \cite[Theorem II]{Hir2} since for any $\eta \in \CR_{m}^{0+}$, $\alpha(A_{\eta},1_n)$ depends only on $m-E_0(\eta)$ ($\eta_1'$ in the notation in loc. cit.).

Therefore, by \cite[Proposition A.5]{Cho2},
\begin{equation*}
\alpha(\pi^{-a_t}B,1_{a_t})=\alpha(\left(\begin{array}{ll}
	\pi 1_{n_{t-1}} & 0\\
	0& 1_{k_t}
	\end{array}\right),1_{k_t})=\prod_{1 \leq m \leq k_t} (1-(-q)^{-m}).
\end{equation*}
This implies that
\begin{equation*}
\alpha(B,B)=q^{a_tn_t^2}\prod_{1 \leq m \leq k_t} (1-(-q)^{-m})\alpha(B_1,B_1),
\end{equation*}
Now, the lemma follows inductively from this.
\end{proof}

Now we computed all constants in \eqref{eq4.1.8}. Since we want to compute $D_{\alpha}$ such that
\begin{equation*}
\dfrac{W'_{0,n}(B,0)}{W_{n,n}(A_n,0)}=\mathlarger{\sum}_{\alpha \in \CR_n^{0+}}D_{\alpha} \dfrac{\alpha(A_{\alpha},B)}{\alpha(A_{\alpha},A_{\alpha})},
\end{equation*}
we need to count all $\lambda$-terms in \eqref{eq4.1.8} such that $\lambda_i^+=\alpha$ for some $i$.

Let us define the following notation. For $\lambda=(\overline{\lambda},\overset{E_1(\lambda)}{\overbrace{1,\dots,1}},\overset{E_0(\lambda)}{\overbrace{0,\dots,0}}) \in D_k$ and $0 \leq p \leq E_1(\lambda)$, we define
\begin{equation*}
\lambda_{p}^{-}=(\overline{\lambda},\overset{E_1(\lambda)-p}{\overbrace{1,\dots,1}},\overset{E_0(\lambda)+p}{\overbrace{0,\dots,0}}),
\end{equation*}
by replacing $p$ 1's by $p$ zeros.

Then we have the following theorem.

\begin{theorem}\label{theorem4.7}
	For $\lambda \in \CR_{2n}^{0+}$, we have
	\begin{equation*}\begin{array}{ll}
			D_{\lambda}&=\dfrac{q^{n^2}\prod_{l=n+1}^{2n}(1-(-q)^{-l})}{\prod_{l=1}^{n}(1-(-q)^{-l})}C_{\lambda}\\\\
			&+\mathlarger{\sum}_{\substack{1 \leq i \leq n\\ \max\lbrace i-E_0(\lambda),0 \rbrace  \leq p \\ \leq \min\lbrace i , E_1(\lambda) \rbrace}}C_{(\lambda_p^-)^{\vee_i}}
			\times \dfrac{\prod_{l=1}^{2n-i}(1-(-q)^{-l})}{q^{-3n^2}(\prod_{l=1}^{n}(1-(-q)^{-l}))^2}\\\\
			&\times
			(-q)^{((n+1)(i-p)+(3n-i+1)(n-i)/2-4n^2)}(-1)^{i-p}\\
			&\times \dfrac{\prod_{j=1}^{n-p}((-q)^n-(-q)^{j-1})}{\prod_{j=1}^{i-p}((-q)^j-1)\prod_{j=1}^{n-i}((-q)^j-1)}.
			\\\\
			&\times
			\dfrac{\prod_{l=1}^{l=E_0(\lambda)}(1-(-q)^{-l})\prod_{l=1}^{l=E_1(\lambda)}(1-(-q)^{-l})}{\prod_{l=1}^{l=E_0(\lambda)-i+p}(1-(-q)^{-l})\prod_{l=1}^{l=E_1(\lambda)-p}(1-(-q)^{-l})} \times q^{p(4n-2E_0(\lambda)-p)}.
		\end{array}
	\end{equation*}

Here, we choose the following convention: For $k \leq 0$, we assume that
\begin{equation*}
	\prod_{l=1}^{l=k} (*)=1.
\end{equation*}

In particular, $D_{\lambda}$ depends only on $n$, $E_0(\lambda)$, $E_1(\lambda)$, and the parity of $\sum_{i}\lambda_i$.
\end{theorem}
\begin{proof}
	We have the following identity from \eqref{eq4.1.8}.
	\begin{equation*}
		\begin{array}{ll}
		D_{\lambda}&=\dfrac{\alpha(1_{2n},1_{2n})C_{\lambda}}{(-q)^{2n^2}W_{n,n}(A_n,0)}\\
		&+\mathlarger{\sum}_{\substack{0 \leq i \leq 2n\\ 0 \leq p \leq 2n}}\dfrac{C_{(\lambda_p^-)^{\vee_i}}}{\alpha(A_{(\lambda_p^-)^{\vee_i}},A_{(\lambda_p^-)^{\vee_i}})}\alpha(1_{2n-i},1_{2n-i})\CK_id_{pi}\dfrac{\alpha(A_{\lambda},A_{\lambda})}{W_{n,n}(A_n,0)}.
		\end{array}
	\end{equation*}

Since $\CK_i=0$ for $i=0$ and $n+1 \leq i \leq 2n$, it suffices to consider $1 \leq i \leq n$. Also, $(\lambda_p^{-})^{\vee_i}$ is defined only when
\begin{equation*}
	\max\lbrace i-E_0(\lambda),0 \rbrace \leq p \leq \min\lbrace i , E_1(\lambda) \rbrace.
\end{equation*}
Therefore, we have
\begin{equation*}
		\begin{array}{l}
		D_{\lambda}=\dfrac{\alpha(1_{2n},1_{2n})C_{\lambda}}{(-q)^{2n^2}W_{n,n}(A_n,0)}\\
		+\mathlarger{\sum}_{\substack{1 \leq i \leq n\\ \max\lbrace i-E_0(\lambda),0 \rbrace  \leq p \\ \leq \min\lbrace i , E_1(\lambda) \rbrace}}\dfrac{C_{(\lambda_p^-)^{\vee_i}}}{\alpha(A_{(\lambda_p^-)^{\vee_i}},A_{(\lambda_p^-)^{\vee_i}})}\alpha(1_{2n-i},1_{2n-i})\CK_id_{pi}\dfrac{\alpha(A_{\lambda},A_{\lambda})}{W_{n,n}(A_n,0)}.
	\end{array}
\end{equation*}

By Lemma \ref{lemma4.6}, we have
\begin{equation*}
	\begin{array}{rl}
		\alpha(A_{\lambda},A_{\lambda})=&q^{(2n-E_0(\lambda))^2}\prod_{l=1}^{l=E_0(\lambda)}(1-(-q)^{-l})\\
		&\times \prod_{l=1}^{l=E_1(\lambda)}(1-(-q)^{-l})\alpha(\pi^{-1}A_{\overline{\lambda}},\pi^{-1}A_{\overline{\lambda}}),\\\\
		
		\alpha(A_{(\lambda_p^-)^{\vee_i}},A_{(\lambda_p^-)^{\vee_i}})=&q^{(2n-E_0(\lambda)-p)^2}\prod_{l=1}^{l=E_0(\lambda)-i+p}(1-(-q)^{-l})\\
		&\times \prod_{l=1}^{l=E_1(\lambda)-p}(1-(-q)^{-l})\alpha(\pi^{-1}A_{\overline{\lambda}},\pi^{-1}A_{\overline{\lambda}}),
		
	\end{array}
\end{equation*}
where
\begin{equation*}
	\lambda=(\overline{\lambda},\overset{E_1(\lambda)}{\overbrace{1,\dots,1}},\overset{E_0(\lambda)}{\overbrace{0,\dots,0}}).
\end{equation*}
Here, we choose the following convention: For $k \leq 0$, we assume that
\begin{equation*}
	\prod_{l=1}^{l=k} (*)=1.
\end{equation*}

Therefore,
\begin{equation*}\begin{array}{l}
	\dfrac{\alpha(A_{\lambda},A_{\lambda})}{\alpha(A_{(\lambda_p^-)^{\vee_i}},A_{(\lambda_p^-)^{\vee_i}})}\\
	=\dfrac{\prod_{l=1}^{l=E_0(\lambda)}(1-(-q)^{-l})\prod_{l=1}^{l=E_1(\lambda)}(1-(-q)^{-l})}{\prod_{l=1}^{l=E_0(\lambda)-i+p}(1-(-q)^{-l})\prod_{l=1}^{l=E_1(\lambda)-p}(1-(-q)^{-l})}\times q^{p(4n-2E_0(\lambda)-p)}.
	\end{array}
\end{equation*}

Also, by Lemma \ref{lemma4.4} and Lemma \ref{lemma4.5}, we have
\begin{equation*}
	\begin{array}{ll} 
	\CK_id_{pi}&=(-q)^{((n+1)(i-p)+(3n-i+1)(n-i)/2-4n^2)}(-1)^{i-p}\\
	&\times \dfrac{\prod_{j=1}^{n-p}((-q)^n-(-q)^{j-1})}{\prod_{j=1}^{i-p}((-q)^j-1)\prod_{j=1}^{n-i}((-q)^j-1)}.
	\end{array}
\end{equation*}

By Lemma \ref{lemma4.6} and \cite[Proposition 3.23]{Cho2}, we have
\begin{equation*}
	\begin{array}{ll}
	\alpha(1_{2n-i},1_{2n-i})=\prod_{l=1}^{2n-i}(1-(-q)^{-l}),\\
	W_{n,n}(A_n,0)=q^{-4n^2}\alpha(\pi A_n,1_n)\alpha(\pi 1_n, \pi 1_n)=q^{-3n^2}(\prod_{l=1}^{n}(1-(-q)^{-l}))^2.
	\end{array}
\end{equation*}

Combining these formulas, we have
\begin{equation*}\begin{array}{ll}
	D_{\lambda}&=\dfrac{q^{n^2}\prod_{l=n+1}^{2n}(1-(-q)^{-l})}{\prod_{l=1}^{n}(1-(-q)^{-l})}C_{\lambda}\\\\
	&+\mathlarger{\sum}_{\substack{1 \leq i \leq n\\ \max\lbrace i-E_0(\lambda),0 \rbrace  \leq p \\ \leq \min\lbrace i , E_1(\lambda) \rbrace}}C_{(\lambda_p^-)^{\vee_i}}
	\times \dfrac{\prod_{l=1}^{2n-i}(1-(-q)^{-l})}{q^{-3n^2}(\prod_{l=1}^{n}(1-(-q)^{-l}))^2}\\\\
	&\times
	(-q)^{((n+1)(i-p)+(3n-i+1)(n-i)/2-4n^2)}(-1)^{i-p}\\
		\end{array}
\end{equation*}
\begin{equation*}\begin{array}{ll}
	&\times \dfrac{\prod_{j=1}^{n-p}((-q)^n-(-q)^{j-1})}{\prod_{j=1}^{i-p}((-q)^j-1)\prod_{j=1}^{n-i}((-q)^j-1)}.
	\\\\
	&\times
	 \dfrac{\prod_{l=1}^{l=E_0(\lambda)}(1-(-q)^{-l})\prod_{l=1}^{l=E_1(\lambda)}(1-(-q)^{-l})}{\prod_{l=1}^{l=E_0(\lambda)-i+p}(1-(-q)^{-l})\prod_{l=1}^{l=E_1(\lambda)-p}(1-(-q)^{-l})}\\\\
	 
	 &\times q^{p(4n-2E_0(\lambda)-p)}.
	\end{array}
\end{equation*}
\end{proof}

\begin{remark}
	By Proposition \ref{proposition3.3} and Proposition \ref{proposition3.4}, we know all $C_{\lambda}$'s explicitly. But, we did not include these formulas in the above theorem since $C_{(0,0,\dots,0)}$'s have different formulas.
	If $E_0(\lambda)+E_1(\lambda) \neq 2n$, i.e., if $\lambda$ does not consist only of $0$'s and $1$'s, then we have
	\begin{equation*}
		\begin{array}{l}
			C_{\lambda}=(-1)^{(\sum_{j}\lambda_j)+1}\prod_{j=1}^{2n-E_0(\lambda)-1}(1-(-q)^j),	\\
		C_{(\lambda_p^{-})^{\vee_i}}=(-1)^{(\sum_{j}\lambda_j)+1-p}\prod_{j=1}^{2n-p-E_0(\lambda)-1}(1-(-q)^j).
		\end{array}
	\end{equation*}

	If $E_0(\lambda)+E_1(\lambda)=2n$, then some $(\lambda_{E_1(\lambda)}^{-})^{\vee_i}$'s are of the form $(0,\dots,0)$, and hence we need to be careful since they are $\alpha'(1_{2n-i},1_{2n-i})/\alpha(1_{2n-i},1_{2n-i})$.
\end{remark}

\subsection{Reformulations of conjectures on $\CN^n(1,2n-1)$}\label{subsection4.2}\quad

In Conjecture \ref{conjecture3.17}, we have a conjectural formula for the arithmetic intersection numbers of special cycles $\langle \CZ(x_1),\dots,\CZ(x_{2n-m}),\CY(y_1),\dots,\CY(y_m) \rangle$ in $\CN^n(1,2n-1)$ for arbitrary $m$. However, in the present paper, we only consider the case where $m=0$. It is because $\langle \CZ(x_1),\dots,\CZ(x_{2n}) \rangle$ is $GL_{2n}(O_E)$-invariant. Therefore, this can be represented as a certain sum of usual representation densities and hence a certain weighted counting of rank $2n$ lattices. If we want to consider arbitrary $m$, we need some parahoric group invariant objects instead of rank $2n$ lattices. We think that we will need all the cases to understand the intersection theory in $\CN^n(1,2n-1)$, but this will be postponed to our later work.  Therefore, in this subsection, we only consider $\langle \CZ(x_1),\dots,\CZ(x_{2n}) \rangle$.

For $\lambda \in \CR_{2n}^{0+}$ and the hermitian matrix $A_{\lambda}$, we consider lattices $L'$ of rank $2n$ with hermitian form $A_{\lambda}$ which we denote by $L' \in A_{\lambda}$. We denote by $1_{L'}$ the characteristic function of $L'$, i.e., $1_{L'}(x_1,\dots,x_{2n})=1$ if and only if $x_i \in L'$ for all $i$.

Also, in Conjecture \ref{conjecture3.17} we have additional terms:
\begin{equation*}
	\mathlarger{\sum}_{0 \leq i \leq n-1} -\dfrac{\beta^0_i}{W_{n,n}(A_n,0)}W_{0,i}(B,0).
\end{equation*}

In terms of usual representation densities, this can be written as
\begin{equation*}
	W_{0,i}(B,0)=(-q)^{-4in}\alpha(A_{(1^i,0^{2n-i})}, B),
\end{equation*}
where $(1^i,0^{2n-i}) \in \CR^{0+}_{2n}$ with $i$ $1$'s and $(2n-i)$ $0$'s.
Therefore, we have
\begin{equation*}
	\begin{array}{l}
	\mathlarger{\sum}_{0 \leq i \leq n-1} -\dfrac{\beta^0_i}{W_{n,n}(A_n,0)}W_{0,i}(B,0)\\
	=-\mathlarger{\sum}_{0 \leq i \leq n-1}(-q)^{-4in}\beta^0_i\dfrac{\alpha(A_{(1^i,0^{2n-i})},A_{(1^i,0^{2n-i})})}{W_{n,n}(A_n,0)}\dfrac{\alpha(A_{(1^i,0^{2n-i})},B)}{\alpha(A_{(1^i,0^{2n-i})},A_{(1^i,0^{2n-i})})}.
	\end{array}
\end{equation*}

Let us define the constants $\Fb_i^0$ as
\begin{equation*}
	\begin{array}{ll}
	\Fb_i^0&:=(-q)^{-4in}\beta^0_i\dfrac{\alpha(A_{(1^i,0^{2n-i})},A_{(1^i,0^{2n-i})})}{W_{n,n}(A_n,0)}\\
	&=\beta_i^0(-q)^{-4in+i^2+3n^2}\dfrac{\prod_{l=1}^{2n-i}((1-(-q)^{-l}))\prod_{l=1}^{i}((1-(-q)^{-l}))}{\prod_{l=1}^{n}((1-(-q)^{-l}))^2}.
	\end{array}
\end{equation*}

Combining these with Conjecture \ref{conjecture3.17} and Theorem \ref{theorem4.7}, we have the following conjecture.

\begin{conjecture}\label{conjecture4.9}
	For a basis $\lbrace x_1, \dots, x_{2n} \rbrace$ of the space of special homomorphisms $\BV$, and special cycles $\CZ(x_1)$, $\dots$, $\CZ(x_{2n})$ in $\CN^n(1,2n-1)$, we have
	\begin{equation*}\begin{array}{l}
		\langle \CZ(x_1),\dots,\CZ(x_{2n})\rangle\\
		=\mathlarger{\sum}_{\lambda \in \CR_{2n}^{0+}} D_{\lambda} \dfrac{\alpha(A_{\lambda},B)}{\alpha(A_{\lambda},A_{\lambda})}-\mathlarger{\sum}_{0 \leq i \leq n-1} \Fb_i^0\dfrac{\alpha(A_{(1^i,0^{2n-i})},B)}{\alpha(A_{(1^i,0^{2n-i})},A_{(1^i,0^{2n-i})})} \\\\
		
		=\mathlarger{\sum}_{\lambda \in \CR_{2n}^{0+}}\mathlarger{\sum}_{L' \in A_{\lambda}} D_{\lambda}1_{L'}(x_1,\dots,x_{2n})
		-\mathlarger{\sum}_{0 \leq i \leq n-1} \mathlarger{\sum}_{L' \in A_{(1^i,0^{2n-i})}} \Fb_i^0 1_{L'}(x_1, \dots, x_{2n}).
		\end{array}
	\end{equation*}
\end{conjecture}

\begin{remark}
	In $\CN^n(1,2n-1)$, let $L$ be a rank $2n$ $O_E$-lattice generated by special homomorphisms $x_1, \dots, x_{2n}$ in $\BV$. Assume that $B$ is the hermitian matrix of $L$. Then, the valuation of the determinant of $B$ and $n+1$ have the same parity. Therefore, in Conjecture \ref{conjecture4.9}, the terms $\alpha(A_{\lambda},B)$ such that
	\begin{equation*}
		\val(\det(A_{\lambda}))=\sum_i \lambda_i \not\equiv n+1 (\Mod 2)
	\end{equation*}
are always equal to 0.
\end{remark}

\begin{remark}
	When $n=1$, the Conjecture \ref{conjecture4.9} can be written as follows.
	In $\CN^1(1,1)$, for linearly independent special homomorphisms $x$ and $y$, we know that determinant of
	\begin{equation*}
		\left( \begin{array}{ll}
			 h(x,x) & h(x,y) \\
			 h(y,x) & h(y,y)
			\end{array}\right)
	\end{equation*}
has an even valuation. Therefore, we only need to consider lattices with hermitian forms $A_{\lambda}$ where $\lambda_1+\lambda_2$ is even. Therefore, for special cycles $\CZ(x)$ and $\CZ(y)$, we have
	\begin{equation*}\begin{array}{ll}
		\langle \CZ(x),\CZ(y) \rangle&=-(q-1) \mathlarger{\sum}_{L' \in A_{(1,1)}} 1_{L'}(x,y)-(q^2-1)\mathlarger{\sum}_{\substack{L' \in A_{(\lambda_1,\lambda_2)}\\ \lambda_1, \lambda_2 \geq 2}} 1_{L'}(x,y)\\
		&\mathlarger{\sum}_{\substack{L' \in A_{(\lambda_1,1)}\\ \lambda_1 \geq 2}} 1_{L'}(x,y)+\mathlarger{\sum}_{\substack{L' \in A_{(\lambda_1,0)}\\ \lambda_1 \geq 2}} 1_{L'}(x,y).
		\end{array}
	\end{equation*}
($A_{(0,0)}$-term cancels).

On the other hand, by \cite[Theorem 3.14]{San2}, we know that for a special homomorphism $y$ with $\val (h(y,y)) \geq 2$, we have that
\begin{equation*}
	\CZ(y)-\CZ(\dfrac{y}{\pi})=\sum_{\substack{y/\pi \in \Lambda\\ \text{type of }\Lambda=2}} \BP_{\Lambda}+\sum_{\substack{y/\pi \in \Lambda\\ \text{type of }\Lambda=0}} \BP_{\Lambda}
\end{equation*}
Here $\Lambda$'s are vertex lattices in a Bruhat-Tits tree.

Also, by \cite[Lemma 2.11]{San}, we know that
\begin{equation*}
	\langle \CZ(x),\BP_{\Lambda})=-q 1_{\Lambda}(\dfrac{x}{\pi}),
\end{equation*}
if $\Lambda$ has type 2, and
\begin{equation*}
	\langle \CZ(x),\BP_{\Lambda})= 1_{\Lambda}(x),
\end{equation*}
if $\Lambda$ has type 0.

Also, note that $\Lambda$ has type 2 if and only if its hermitian form is equivalent to $A_{(-1,-1)}$ and $\Lambda$ has type 0 if and only if its hermitian form is equivalent to $A_{(0,0)}$.

Combining these, we have the following equality.

	\begin{equation*}\begin{array}{ll}
		\langle \CZ(x),\CZ(y)-\CZ(\dfrac{y}{\pi}) \rangle&=-(q-1) \mathlarger{\sum}_{L' \in A_{(1,1)}} \lbrace  1_{L'}(x,y)-1_{L'}(x,\dfrac{y}{\pi}) \rbrace\\
		&-(q^2-1)\mathlarger{\sum}_{\substack{L' \in A_{(\lambda_1,\lambda_2)}\\ \lambda_1, \lambda_2 \geq 2}} \lbrace  1_{L'}(x,y)-1_{L'}(x,\dfrac{y}{\pi}) \rbrace\\
		&\mathlarger{\sum}_{\substack{L' \in A_{(\lambda_1,1)}\\ \lambda_1 \geq 2}} \lbrace  1_{L'}(x,y)-1_{L'}(x,\dfrac{y}{\pi}) \rbrace\\
		&+\mathlarger{\sum}_{\substack{L' \in A_{(\lambda_1,0)}\\ \lambda_1 \geq 2}} \lbrace  1_{L'}(x,y)-1_{L'}(x,\dfrac{y}{\pi}) \rbrace
		\end{array}
	\end{equation*}
\begin{equation*}
	\begin{array}{ll}
	\quad\quad\quad\quad\quad\quad\quad\quad\quad\quad\quad	&=-q\mathlarger{\sum}_{L' \in A_{(-1,-1)}} 1_{L'}(\dfrac{x}{\pi},\dfrac{y}{\pi})+\mathlarger{\sum}_{L' \in A_{(0,0)}} 1_{L'}(x,\dfrac{y}{\pi})\\\\
		&=-q\mathlarger{\sum}_{L' \in A_{(1,1)}} 1_{L'}(x,y)+\mathlarger{\sum}_{L' \in A_{(0,0)}} 1_{L'}(x,\dfrac{y}{\pi})\\
	\end{array}
\end{equation*}

Since this holds only for $y$ with $\val(h(y,y))\geq 2$, this is not $GL_2(O_E)$-invariant. We think that we will need all Iwahoric invariant objects to understand this equality. Therefore, this will be postponed to our future work.
\end{remark}

\end{document}